\documentclass[a4paper,12pt]{article}

\usepackage{lmodern}  
\usepackage[T1]{fontenc}
\usepackage[latin1]{inputenc}

\usepackage{textcomp} 
\usepackage{graphics} 
\usepackage{amsmath} 
\usepackage{amsthm} 
\usepackage{amssymb} 
\usepackage{amsfonts}
\usepackage{mathtools}
\usepackage{thmtools} 
\usepackage{mathrsfs} 
\usepackage{MnSymbol} 
\usepackage{extpfeil} 
\usepackage{verbatim} 
\usepackage{longtable} 
\usepackage[table]{xcolor} 
\usepackage{afterpage} 
\usepackage{xparse} 
\usepackage{mathscinet} 
\usepackage{adjustbox} 
\usepackage[obeyDraft]{todonotes} 

\usepackage[final,                  
            pdftex,                 
            pdfpagelabels,
            pdfstartview = {FitH},  
            bookmarks,              
            colorlinks,             
            plainpages = false,     
            linktoc=all,            
            linkcolor=black,        
            citecolor=blue,         
            urlcolor=blue,          
            filecolor=black         
            ]{hyperref}

\usepackage[citestyle=alphabetic,   
            bibstyle=alphabetic,    
            maxbibnames=50,         
            maxcitenames=3,         
            autocite=inline,        
            block=space,            
            hyperref=true,          
            backref=true,           
            backrefstyle=three,     
            date=short,             
            arxiv=pdf,              
            isbn=false,             
            url=false,              
            doi=true,               
            eprint=true,            
            firstinits=true,        
            backend=bibtex8,
            ]{biblatex}

\bibliography{faithfulness}

\hypersetup{pdftitle = {The 2-braid group and Garside normal form},
            pdfauthor = {Lars Thorge Jensen},
            pdfkeywords = {categorification} {braid groups} {Rouquier complex} {Soergel bimodules} {Garside normal form}
                           {2-braid group},
            pdfdisplaydoctitle      
            }
            
\usepackage[arrow, matrix, curve]{xy} 
\usepackage{tikz} 
\usetikzlibrary{arrows, fit, calc, shapes.misc, shapes.geometric, decorations.markings}
\usepackage{pgfplots}

\allowdisplaybreaks 
            
\declaretheoremstyle[
spaceabove=\topsep, spacebelow=\topsep,
headfont=\normalfont\bfseries,
notefont=\mdseries, notebraces={(}{)},
bodyfont=\itshape,
postheadspace=\newline
]{break}

\declaretheoremstyle[
spaceabove=\topsep, spacebelow=\topsep,
headfont=\normalfont\bfseries,
notefont=\mdseries, notebraces={}{},
bodyfont=\itshape,
postheadspace=\newline
]{refbreak}

\declaretheorem[title=Theorem, style=plain, numberwithin=section]{thm}
\declaretheorem[title=Proposition, style=plain, numberlike=thm]{prop}
\declaretheorem[title=Lemma, style=plain, numberlike=thm]{lem}
\declaretheorem[title=Corollary, style=plain, numberlike=thm]{cor}

\declaretheorem[title=Theorem, style=break, numberlike=thm]{thmlab}

\declaretheorem[title=Lemma, style=break, numberlike=thm]{lemlab}
\declaretheorem[title=Corollary, style=break, numberlike=thm]{corlab}
\declaretheorem[title=Conjecture, style=break, numberlike=thm]{conj}

\declaretheorem[title=Proposition, style=refbreak, numberlike=thm]{propreflab}
\declaretheorem[title=Lemma, style=refbreak, numberlike=thm]{lemreflab}

\declaretheorem[title=Definition, style=definition, numberlike=thm]{defn}

\declaretheorem[title=Remark, style=remark, numberlike=thm]{remark}

\declaretheorem[title=Claim, style=plain, numbered=no]{claim}
\declaretheorem[title=Example, style=remark, numberlike=thm]{ex}

\usepackage[capitalize]{cleveref} 
                          
\crefname{thm}{Theorem}{Theorems}
\crefname{prop}{Proposition}{Propositions}
\crefname{lem}{Lemma}{Lemmata}
\crefname{cor}{Corollary}{Corollaries}
\crefname{rem}{Reminder}{Reminders}
\crefname{defn}{Definition}{Definitions}

\crefname{thmlab}{Theorem}{Theorems}
\crefname{proplab}{Proposition}{Propositions}
\crefname{lemlab}{Lemma}{Lemmata}
\crefname{corlab}{Corollary}{Corollaries}
\crefname{remlab}{Reminder}{Reminders}
\crefname{conj}{Conjecture}{Conjectures}

\crefname{thmreflab}{Theorem}{Theorems}
\crefname{propreflab}{Proposition}{Propositions}
\crefname{lemreflab}{Lemma}{Lemmata}
\crefname{correflab}{Corollary}{Corollaries}
\crefname{remreflab}{Reminder}{Reminders}
\crefname{conjref}{Conjecture}{Conjectures}

\crefname{remark}{Remark}{Remarks}
\crefname{claim}{Claim}{Claims}
\crefname{ex}{Example}{Examples}

\CompileMatrices

\entrymodifiers={+!!<0pt,\fontdimen22\textfont2>} 

\def\clap#1{\hbox to 0pt{\hss#1\hss}}

\makeatletter
\def\underbracket{%
    \@ifnextchar[{\@underbracket}{\@underbracket [\@bracketheight]}%
}
\def\@underbracket[#1]{%
    \@ifnextchar[{\@under@bracket[#1]}{\@under@bracket[#1][0.4em]}%
}
\def\@under@bracket[#1][#2]#3{
    \mathop{\vtop{\m@th \ialign {##\crcr $\hfil \displaystyle {#3}\hfil $%
    \crcr \noalign {\kern 3\p@ \nointerlineskip }\upbracketfill {#1}{#2}
    \crcr \noalign {\kern 3\p@ }}}}\limits}

\def\upbracketfill#1#2{$\m@th \setbox \z@ \hbox {$\braceld$}
    \edef\@bracketheight{\the\ht\z@}\bracketend{#1}{#2}
    \leaders \vrule \@height #1 \@depth \z@ \hfill
    \leaders \vrule \@height #1 \@depth \z@ \hfill \bracketend{#1}{#2}$}

\def\bracketend#1#2{\vrule height #2 width #1\relax}
\makeatother

\makeatletter
\def\thmt@refnamewithcomma #1#2#3,#4,#5\@nil{%
  \@xa\def\csname\thmt@envname #1utorefname\endcsname{#3}%
  \ifcsname #2refname\endcsname
    \csname #2refname\expandafter\endcsname\expandafter{\thmt@envname}{#3}{#4}%
  \fi
}
\makeatother

\makeatletter
\newcommand*\rel@kern[1]{\kern#1\dimexpr\macc@kerna}
\newcommand*\widebar[1]{%
  \begingroup
  \def\mathaccent##1##2{%
    \rel@kern{0.8}%
    \overline{\rel@kern{-0.8}\macc@nucleus\rel@kern{0.2}}%
    \rel@kern{-0.2}%
  }%
  \macc@depth\@ne
  \let\math@bgroup\@empty \let\math@egroup\macc@set@skewchar
  \mathsurround\z@ \frozen@everymath{\mathgroup\macc@group\relax}%
  \macc@set@skewchar\relax
  \let\mathaccentV\macc@nested@a
  \macc@nested@a\relax111{#1}%
  \endgroup
}
\makeatother

\makeatletter
\newcommand{\subjclass}[2][1991]{%
  \let\@oldtitle\@title%
  \gdef\@title{\@oldtitle\footnotetext{#1 \emph{Mathematics subject classification.} #2}}%
}
\newcommand{\keywords}[1]{%
  \let\@@oldtitle\@title%
  \gdef\@title{\@@oldtitle\footnotetext{\emph{Key words and phrases.} #1.}}%
}
\makeatother

\DeclareMathOperator{\Hom}{Hom}

\DeclareMathOperator{\im}{im}
\DeclareMathOperator{\id}{id}
\DeclareMathOperator{\inkl}{\iota}
\DeclareMathOperator{\pr}{pr}

\newcommand{\cat}[1]{\ensuremath{\mathcal{#1}}}

\newcommand{\ideal}[1]{\ensuremath{\mathfrak{#1}}}

\newcommand{\obj}[1]{\ensuremath{\in \cat{#1}}}
\newcommand{\sbim}{\ensuremath{\cat{B}}}

\newcommand{\bsbim}{\ensuremath{\cat{BS}}}

\newcommand{\rgrbim}{\ensuremath{R\text{-grmod-}R}}

\newcommand{\Kb}{\ensuremath{K^{b}(\sbim)}}

\newcommand{\br}[1][]{\ensuremath{Br_{#1}}}
\newcommand{\brm}[1][]{\ensuremath{Br_{#1}^{+}}}
\newcommand{\brred}[1][]{\ensuremath{Br_{#1}^{\text{red}}}}
\newcommand{\brcat}{\ensuremath{2-\cat{B}r}}

\newcommand{\heck}[1][]{\ensuremath{\mathbf{H}_{#1}}}

\newcommand{\std}[1]{\ensuremath{H_{#1}}}
\newcommand{\kl}[1]{\ensuremath{\underline{H}_{#1}}}
\newcommand{\desc}[1]{\ensuremath{\mathcal{#1}}}

\newcommand{\leqcell}[1]{\ensuremath{\underset{#1}{\leq}}}
\newcommand{\lcell}[1]{\ensuremath{\underset{#1}{<}}}

\newcommand{\cellUniqueRex}{\ensuremath{C}}
\newcommand{\cellNonUniqueRex}{\ensuremath{\widetilde{C}}}
\newcommand{\lcellFixedRightDesc}[1]{\ensuremath{\cellUniqueRex}_{#1}}

\newcommand{\elFixedLeftDesc}[1]{\ensuremath{\prescript{}{#1}{w}}}

\newcommand{\catFixedRightDesc}[1]{\ensuremath{\cat{C}_{#1}}}

\newcommand{\Z}{\ensuremath{\mathbb{Z}}}
\newcommand{\N}{\ensuremath{\mathbb{N}}}
\newcommand{\R}{\ensuremath{\mathbb{R}}}

\newcommand{\defeq}{\ensuremath{\coloneqq}}
\newcommand{\pH}[1]{\ensuremath{\prescript{p}{}{\mathcal{H}}^{#1}}}
\newcommand{\G}[1]{\ensuremath{\mathfrak{K}^0(\cat{#1})}}
\newcommand{\Galg}[1]{\ensuremath{\mathfrak{K}^0(#1)}}

\newcommand{\tauleq}[1]{\ensuremath{\prescript{p}{}{\tau}_{\leqslant #1}}}
\newcommand{\taul}[1]{\ensuremath{\prescript{p}{}{\tau}_{< #1}}}
\newcommand{\taugeq}[1]{\ensuremath{\prescript{p}{}{\tau}_{\geqslant #1}}}
\newcommand{\taug}[1]{\ensuremath{\prescript{p}{}{\tau}_{> #1}}}

\newcommand{\KCleq}[1]{\ensuremath{\prescript{p}{}{K^{b}}(\cat{C})^{\leqslant #1}}}

\newcommand{\KCgeq}[1]{\ensuremath{\prescript{p}{}{K^{b}}(\cat{C})^{\geqslant #1}}}

\newcommand{\Address}{
  \bigskip{\footnotesize

  \textsc{Max Planck Institute for Mathematics, Vivatsgasse 7,
  53111 Bonn}\par\nopagebreak
  \textit{E-mail address}, Lars~Thorge~Jensen: \texttt{ltjensen@mpim-bonn.mpg.de}
}}

\tikzset{%
  highlight/.style={rectangle,rounded corners,fill=red!15,draw=red,
    fill opacity=0.5,thick},
  bendBelow/.style={bend left=70, looseness=2},
  bendAbove/.style={bend right=70, looseness=2},
  object/.style={circle, fill, inner sep=1.5pt, outer sep=0mm},
  labeling/.style={outer sep=0mm, inner sep=0mm},
  1morph/.style={->, shorten >= 0.5pt, >=stealth'},
  2morph/.style={-implies,double,double equal sign distance,
                 shorten >=2pt, shorten <=3pt},
  spot/.style={color=black, thin, dashed},
  sline/.style={color=blue, line width=2pt},
  tline/.style={color=red, line width=2pt},
  uline/.style={color=green, line width=2pt},
  randline/.style={color=pink, line width=2pt},
  sdot/.style={color=blue, thin, fill},
  tdot/.style={color=red, thin, fill},
  udot/.style={color=green, thin, fill},
  randdot/.style={color=pink, thin, fill}
}

\newcommand{\tikzmark}[2]{\tikz[overlay,remember picture,
  baseline=(#1.base)] \node (#1) {#2};}

\newcommand{\DrawBox}[1][submatrix]{%
    \tikz[overlay,remember picture]{
    \node[highlight,fit=(left.north west) (right.south east)] (#1) {};}
}

\newcommand{\BigFig}[1]{\parbox{12pt}{\Huge #1}}
\newcommand{\BigZero}{\BigFig{0}}

\begin{document}

\title{The 2-braid group and \\ Garside normal form}
\author{Lars Thorge Jensen}
\subjclass[2000]{Primary 20F36, 20J05}
\keywords{categorification, braid groups, Rouquier complex, Soergel bimodules, 
         Garside normal form, 2-braid group}
\date{} 

\maketitle


\begin{abstract}
We investigate the relation between the Garside normal form for positive braids and
the $2$-braid group defined by Rouquier. Inspired by work of Brav and Thomas we show that the Garside
normal form is encoded in the action of the $2$-braid group on a certain categorified 
left cell module. This allows us to deduce the faithfulness of the $2$-braid group 
in finite type. We also give a new proof of Paris' theorem that the canonical map 
from the generalized braid monoid to its braid group is injective in arbitrary 
type.
\end{abstract}

\section{Introduction}
\label{secIntro}
The braid group is ubiquitous not only in knot theory, but also in topology,
algebraic geometry and representation theory. Experience tells
that the action of the Coxeter group (or its corresponding Hecke algebra) on 
the level of Grothendieck groups can often be upgraded to a braid group action 
on the underlying category. Examples of this phenomenon in representation theory
can be seen in \cite{Ro}, \cite{BR} and \cite{CR}. Rouquier suggests that not
only the self-equivalences of the action are important, but also the morphisms
between them possess some interesting structure. Thus he introduced in \cite{Ro}
the $2$-braid group as a concrete home to study these morphisms. The $2$-braid
group lives in the homotopy category of Soergel bimodules, but has many other
incarnations (translation functors on Bernstein-Gelfand-Gelfand
category $\mathcal{O}$, convolution functors, spherical twists, \dots) as well.

The $2$-braid group is a fundamental mathematical object. Its importance was 
underlined by its applications in categorified link invariants. Just as many
knot invariants factor over the braid group, the $2$-braid group can be used
to construct triply-graded HOMFLY-PT homology (see \cite{Kh}).

The more natural a categorification is, the more structure of the underlying categorified mathematical
object it reflects. Motivated by work of Brav and Thomas we looked for shadows of the 
Garside normal form of the positive braid monoid in the $2$-braid group. It is 
remarkable that on the categorical level the Garside normal form only becomes apparent
after acting on a certain categorified left cell module for the Hecke algebra
(see \cref{thmGarsideAct}). From this we can deduce
the faithfulness of the $2$-braid group in finite type (see \cref{corFaithfulness}) 
as conjectured by Rouquier in \cite{Ro}. Following Rouquier's philosophy a categorified 
braid group action should encode 
a lot of information about the category acted upon. For this purpose proving 
the faithfulness of the $2$-braid group is a basic question.The faithfulness 
follows in type $A$ from work by Khovanov and Seidel (see \cite{KS}) and
in simply-laced, finite type from results by Brav and Thomas (see \cite{BT}). 
The shadows of the Garside normal form also enable us to give a new proof of 
Paris' theorem that the canonical map from the generalized braid monoid to 
its braid group is an injection in arbitrary type (see \cite{Pa} and \cref{corCanInject}).

The following example illustrates that the categorified braid group action
contains strictly more information than its decategorification.
On the one hand, the categorified left cell module we construct admits a faithful
braid group action (see \cref{thmActFaithful}). On the other hand, it gives a 
categorification of a twisted reduced Burau representation in type $A_{n-1}$ 
(see \cref{exDecatNonFaithful}) which is known not to be faithful for 
$n \geqslant 5$ (see \cite{Bi}).

\subsection {Structure of the paper}
\begin{description}
   \item[\Crefrange{secHeckeAlg}{secGenBraidGroups}] We introduce notation and recall important results about
         the Hecke algebra, cells with respect to the Kazhdan-Lusztig basis, Soergel bimodules, generalized
         braid groups and the $2$-braid group.
   \item[\Cref{secHomFormCons}] Using Soergel's $\Hom$-formula we study the existence of degree $1$ morphisms
         between indecomposable Soergel bimodules and rewrite the multiplication formula for the Kazhdan-Lusztig
         basis in our setting.
   \item[\Cref{secCellModCat}] We show that cell modules of the Hecke algebra can be categorified by mimicking
         the construction on the category of Soergel bimodules.
   \item[\Cref{secPervDef}] We introduce the perverse filtration on the homotopy category of Soergel bimodules
         and recall some important results.
   \item[\Cref{secProof}] After introducing the important notion of an anchor we prove our main results.
\end{description}

\subsection{Acknowledgement}

I would like to thank my advisor, Geordie Williamson, for his support and
encouragement. I am grateful to Hanno Becker for very valuable discussions.

\subsection{The Hecke algebra and cells}
\label{secHeckeAlg}

Let $(W,S)$ be a \emph{Coxeter system}, i.e. $W$ is a group together with a set of generators $S$
admitting a particularly nice presentation:
\begin{equation*}
   W = \; \langle s \in S \; \vert \; \underbrace{sts\dots}_{m_{s,t} \text{ terms}} = 
      \underbrace{tst\dots}_{m_{s,t} \text{ terms}} \text{, } s^2 = 1 \rangle
\end{equation*}
where $m_{s,t} \geqslant 2$ is the order of $st$ for $s \neq t \in S$. For $w \in W$ 
denote the left (resp. right) descent set of $w$ by $\desc{L}(w) = \{ s \in S \; \vert \; sw < w \}$ 
(resp. $\desc{R}(w) = \{ s \in S \; \vert \; ws < w \}$).

Let $\heck[(W, S)]$ be the corresponding \emph{Hecke algebra} over $\Z[v, v^{-1}]$ which
we will also denote by $\heck$ if there is no danger of confusion.

Denote by $\{\std{w}\}_{w \in W}$ the \emph{standard} and by $\{\kl{w}\}_{w\in W}$ the 
\emph{Kazhdan-Lusztig basis} in Soergel's normalization (see \cite{S2}). Since we are not working with 
the usual normalization, let us state the relations for the standard basis:
\begin{alignat*}{2}
   \std{s}^2 &= (v^{-1} - v) \std{s} + 1  \qquad && \text{for all } s \in S \text{,} \\
   \underbrace{\std{s} \std{t} \std{s} \dots}_{m_{s, t} \text{ terms}} &= 
      \underbrace{\std{t} \std{s} \std{t} \dots}_{m_{s,t} \text{ terms}} 
      && \text{for all } s \neq t \in S \text{.}
\end{alignat*}

Write $\kl{x} = \sum_{y \leqslant x} h_{y, x} \std{y}$ where $h_{y, x} \in \Z[v]$
are the \emph{Kazhdan-Lusztig polynomials} up to simple renormalization (see \cite{S2}). 
From the defining property of the Kazhdan-Lusztig basis, one sees immediately 
$h_{x, x} = 1$ and $h_{y, x} \in v\Z[v]$ for all $y < x$. Define $\mu(y, x)$ 
for $y \leqslant x \in W$ as the coefficient of $v$ in $h_{y, x}$. We extend 
this definition of $\mu$ as follows. Set $\mu(x, y) \defeq \mu(y, x)$ if $y < x$
and $\mu(x, y) = 0$ if $x$ and $y$ are incomparable in the Bruhat order.

There is a unique $\Z$-linear involution $\widebar{(-)}$ on $\heck$ satisfying
$\widebar{v} = v^{-1}$ and $\widebar{\std{s}} = \std{s}^{-1}$ for $s \in S$ and 
thus $\widebar{\std{x}} = \std{x^{-1}}^{-1}$. 
The Kazhdan-Lusztig basis element $\kl{x}$ is the unique element in $\std{x} + 
\sum_{y < x} v\Z[v] H_y$ that is invariant under $\widebar{(-)}$. 

Moreover, there is a $\Z[v, v^{-1}]$-linear anti-involution $\iota$ on 
$\heck$ satisfying $\iota(\std{s}) = \std{s}$ for $s \in S$ and thus $\iota(\std{x}) = \std{x^{-1}}$
and a $\Z$-linear anti-involution $\omega$ on $\heck$ satisfying
$\omega(v) = v^{-1}$ and $\omega(\kl{s}) = \kl{s}$ for all $s \in S$.

Recall that a \emph{trace} on $\heck$ is a $\Z[v, v^{-1}]$-linear map
$\varepsilon: \heck \rightarrow \Z[v, v^{-1}]$ satisfying $\varepsilon(h h') =
\varepsilon(h' h)$ for all $h, h' \in \heck$. A calculation shows that the $\Z[v, v^{-1}]$-linear map
$\varepsilon$ satisfying $\varepsilon(\std{x}) = \delta_{x, 1}$, where $\delta_{x, 1}$ is the Kronecker symbol,
is a trace. This map is called the \emph{standard trace}.

Using the standard trace and the $\Z$-linear anti-involution $\omega$, we can now define
the \emph{standard pairing} $(-, -): \heck \times \heck \rightarrow \Z[v, v^{-1}]$ via
$(h, h') = \varepsilon(\omega(h)h')$ for all $h, h'\in \heck$. Note that this
pairing is $\Z[v, v^{-1}]$-semilinear, i.e. $(v^{-1}h, h') = (h, vh') = v(h, h')$
for all $h, h' \in \heck$. In addition, $\kl{s}$ is self-biadjoint with respect to
this pairing: $(\kl{s} h, h') = (h, \kl{s} h')$ and $(h \kl{s}, h') = (h, h' \kl{s})$ 
for all $h, h' \in \heck$.

Denote the \emph{left} (resp. \emph{right} or \emph{two-sided}) \emph{cell preorder} with respect
to the Kazhdan-Lusztig basis by $\leqcell{L}$ (resp. $\leqcell{R}$ or $\leqcell{LR}$).
Recall that $\leqcell{L}$ is generated by the relation $v \leqcell{L} w$ for $v,w \in W$ 
if $\kl{v}$ occurs with non-zero coefficient in $h \kl{w}$ for some $h \in \heck$. Observe
that this definition can be generalized to give a preorder on the indexing set of a
basis of any based $R$-algebra.

For $w \in W$ let $\heck( \leqcell{L} w)$ (resp. $\heck( \lcell{L} w)$) be the 
$\Z[v, v^{-1}]$-span of all basis elements $\kl{v}$ such that $v \leqcell{L} w$ 
(resp. $v \lcell{L} w$). Use a similar notation for any left cell instead of $w$ and for the other
cell preorders. We will be interested in the best understood situation, namely 
a left cell inside the two-sided cell of all non-trivial elements with a unique 
reduced expression:

Let $\cellUniqueRex$ be the set of all elements in $W \setminus \{id\}$ with a unique reduced expression and set 
$ \lcellFixedRightDesc{s} \defeq \{ w \in \cellUniqueRex \; \vert \; ws < w\}$. Obviously the sets
$\lcellFixedRightDesc{s}$ for $s \in S$ form a partition of $\cellUniqueRex$.

\begin{propreflab}[{\cite[Proposition 3.8]{Lu1}}]
   Assume $(W, S)$ to be irreducible. Then:
   \begin{enumerate}
      \item $\lcellFixedRightDesc{s}$ is a left cell in $\heck$ for all $s \in S$.
      \item $\cellUniqueRex$ is a two-sided cell in $\heck$.
   \end{enumerate}
\end{propreflab}

To visualize the left cell $\lcellFixedRightDesc{s}$ for $s \in S$ define an undirected 
graph $\Gamma_s$: The vertex set is given by $\lcellFixedRightDesc{s}$ and the edge 
set contains an edge $\{x, y\}$ if $x^{-1}y$ lies in $S$. Define a map 
$\pi_s: \lcellFixedRightDesc{s} \rightarrow S$ by sending an element 
$w \in \lcellFixedRightDesc{s}$ to the unique element in its left descent set 
$\desc{L}(w)$.

\begin{lemreflab}[{\cite[Proposition 3.8]{Lu1}}]
   \label{lemGammaSFacts}
   Assume $(W, S)$ to be irreducible. For any $s \in S$:
   \begin{enumerate}
      \item \label{itmGammaSTree}
            $\Gamma_s$ is a tree.
      \item \label{itmPiSBij}
            $\pi_s$ defines an isomorphism between $\Gamma_s$ and the Coxeter graph of $(W, S)$ if and only
            if the latter is a simply-laced tree.
      \item \label{itmGammaSEdge}
            For $x, y \in \lcellFixedRightDesc{s}$ the subset $\{x, y\}$ is an edge of $\Gamma_s$ if and only if
            $\desc{L}(x) \neq \desc{L}(y)$ and $\mu(x, y) \neq 0$. In this case we have $\mu(x,y) = 1$.
   \end{enumerate}
\end{lemreflab}

\subsection{Soergel bimodules}

Let $\mathfrak{h} = \bigoplus_{s \in S} \R \alpha_s^{\vee}$ be the \emph{reflection representation}
of $(W, S)$ and define the simple roots $\{ \alpha_s \; \vert \; s \in S \} \subset \mathfrak{h}^{\ast}$ via:
\begin{equation} \label{eqnRootCorootRel}
   \langle \alpha_s^{\vee}, \alpha_t \rangle = -2\cos(\frac{\pi}{m_{s, t}})
\end{equation}
(This gives a \emph{symmetric realization} in the sense of \cite[Definition 3.1]{EW2}.)

Denote by $R = S(\mathfrak{h}^{\ast})$ the symmetric algebra on $\mathfrak{h}^{\ast}$, 
viewed as a graded algebra with $\deg(\mathfrak{h}^{\ast}) = 2$. Since 
$W$ acts on $\mathfrak{h}^{\ast}$ via the contragredient representation 
($s(\gamma) = \gamma - \langle \alpha_s^{\vee}, \gamma \rangle \alpha_s$ for all 
$\gamma \in \mathfrak{h}^{\ast}$), we can extend this to an action of $W$ on $R$ 
by degree-preserving algebra automorphisms.

Denote by $\rgrbim$ the abelian, monoidal category of $\Z$-graded $R$-bimodules 
that are finitely generated as left and as right $R$-modules with degree-preserving
bimodule homomorphisms as morphisms. Given a graded $R$-bimodule $M = \bigoplus_{i \in \Z} 
M^{i}$ we denote by $M(1)$ the $R$-bimodule with the grading shifted down by one: 
$M(1)^i = M^{i+1}$. For any two graded $R$-bimodules $M$ and $N$ denote by 
$\Hom^{\bullet}(M, N) \defeq \bigoplus_{n \in \Z}\Hom_{\rgrbim}(M, N(n))$ the bimodule 
homomorphisms from $M$ to $N$ of all degrees.

For $s \in S$ let $R^s \subseteq R$ be the subring of invariants under the action of $s$
and define the $R$-bimodule $B_s = R \otimes_{R^s} R(1)$.

The \emph{category of Bott-Samelson bimodules}, denoted by $\bsbim$, is the full 
additive, monoidal subcategory of $\rgrbim$ generated by the $B_s$ for $s \in S$ 
and their grading shifts. For an expression $\underline{w} = s_1 s_2 \cdots s_k$ 
with $s_i \in S$ for all $1 \leqslant i \leqslant k$ the graded 
$R$-bimodule $B_{\underline{w}} = B_{s_1} \otimes B_{s_2} \otimes \cdots 
\otimes B_{s_k}$ is called a \emph{Bott-Samelson bimodule}. 
We usually omit all tensor products and thus $B_{\underline{w}}$ is 
written as $B_{s_1} B_{s_2} \cdots B_{s_k}$. Define \emph{the category 
of Soergel bimodules} $\sbim$ to be the Karoubi envelope of $\bsbim$. In other words, an 
indecomposable Soergel bimodule is a direct summand of a shifted Bott-Samelson bimodule 
and the morphisms between Soergel bimodules are degree-preserving. The following result is
well known (see \cite[Lemma 6.24]{EW2}):

\begin{lem}
   \label{lemKrullSchmidt}
   The category of Soergel bimodules $\sbim$ is a Krull-Schmidt category with \
   finite dimensional $\Hom$-spaces.
\end{lem}

Recall that for an essentially small, additive, monoidal category $\cat{C}$, its 
split Grothendieck group $\G{C}$ is an associative, unital ring with 
multiplication given by $[M] [N] = [M \otimes N]$ for $M, N \obj{C}$ 
and the class of the monoidal identity $[1_{\cat{C}}]$ as unit. Applying this to 
the category of Soergel bimodules, we see that $\G{\sbim}$ is an associative, 
unital $\Z$-algebra. We can equip it with the structure of a $\Z[v, v^{-1}]$ algebra
by defining $v [B] \defeq [B(1)]$ for all $B \in \sbim$. Soergel proves that the category 
of Soergel bimodules gives a categorification of the Hecke algebra $\heck$ \cite[Theorem 1.10]{S4}
and describes the graded rank of the homomorphism space between two Soergel 
bimodules (see \cite[Theorem 5.15]{S4}):

\begin{thmlab}[Soergel's categorification Theorem]
   \label{thmSoergelCat}
   There is a unique isomorphism of $\Z[v, v^{-1}]$-algebras:
   \begin{align*}
      \varepsilon: \; &\heck \overset{\cong}{\longrightarrow} \G{\sbim} \\
      & \kl{s} \mapsto [B_s] \text{.}
   \end{align*}
\end{thmlab}

\begin{thmlab}[Soergel's $\Hom$-formula]
   \label{thmSoergelHom}
   Given any two Soergel bimodules $B$ and $B'$, the morphism space 
   $\Hom^{\bullet}(B, B')$ is free as a left (resp. right) $R$-module. 
   Moreover, its graded rank is given by $(\varepsilon^{-1} [B], \varepsilon^{-1} [B'] )$
   where $(-, -)$ denotes the standard pairing on the Hecke algebra.
\end{thmlab}

Using his $\Hom$-formula, Soergel obtains a classification of the indecomposable 
Soergel bimodules in \cite[Theorem 6.14, (1) and (2)]{S4}:

\begin{thm}
   \label{thmIndecompSBimClass}
   Given any reduced expression $\underline{w}$ of $w \in W$, the Bott-Samelson $B_{\underline{w}}$
   contains up to isomorphism a unique indecomposable summand $B_w$ which does not occur in $B_{\underline{y}}$
   for any reduced expression $\underline{y}$ of $y \in W$ with $l(y) < l(w)$. In addition, $B_w$ does 
   not depend up to isomorphism on the reduced expression $\underline{w}$.
   A complete set of representatives of the isomorphism classes of all indecomposable Soergel bimodules is given by:
   \[ \{ B_w(m) \; \vert \; w \in W \text{ and } m \in \Z \}. \]
\end{thm}

Using \cref{thmIndecompSBimClass} it follows that $\{[B_w] \; \vert \; w \in W\}$ is 
a $\Z[v, v^{-1}]$ basis of $\G{\sbim}$. It is a natural question what this basis is.
Soergel explicitly constructs an inverse to $\varepsilon$, called the character 
map $ch:\; \G{\sbim} \longrightarrow \heck$, and conjectures that the basis of 
the indecomposable Soergel bimodules $\{[B_w] \; \vert \; w \in W \}$ in $\G{\sbim}$ corresponds
to the Kazhdan-Lusztig basis in $\heck$. Elias and Williamson have recently proven 
Soergel's conjecture for reflection faithful realizations over $\R$ with linear independent 
sets of simple roots and co-roots satisfying a positivity condition. Libedinsky 
showed in \cite{Li2} that their results extend to the reflection representation:

\begin{thmlab}[Soergel's conjecture]
   \label{thmSoergelConj}
   For all $w \in W$ we have $\varepsilon(\kl{w}) = [B_w]$.
\end{thmlab}

This has the following important consequence for us (see \cite[top of page 15]{EW1};
note that in \cite{EW1} a different pairing is used):
\begin{cor}
   \label{corHomNonNegDeg}
   For all $x, y \in W$ the homomorphism space $\Hom^{\bullet}(B_x, B_y)$ 
   is concentrated in non-negative degrees and $\dim\Hom_{\sbim}(B_x, B_y) = \delta_{xy}$
\end{cor}
\begin{proof}
   Soergel's $\Hom$ formula together with Soergel's conjecture imply that the graded 
   rank of $\Hom_{\sbim}^{\bullet}(B_x, B_y)$ is given by
   \begin{equation}
      \label{eqnHomCalc}
      \begin{array}{r@{}l}
         (\kl{x}, \kl{y}) &{}= \varepsilon(\kl{x^{-1}} \kl{y}) \\
            &{}= \varepsilon \left( \left(\sum_{z^{-1} \leqslant x^{-1}} h_{z^{-1}, x^{-1}} \std{z^{-1}} \right)
               \left( \sum_{z \leqslant y} h_{z, y} \std{z} \right) \right) \\
            &{}= \sum_{\substack{z \in W \text{ s.t.} \\ z^{-1} \leqslant x^{-1} \text{ and} \\ z \leqslant y} }
                  h_{z^{-1}, x^{-1}} h_{z, y} \\
            &{}= \sum_{\substack{z \in W \text{ s.t.} \\ z \leqslant x, y} }
                  h_{z, x} h_{z, y} \enspace \in \enspace
                  \begin{cases}
                     v \N[v] & \text{if } x \neq y \\
                     1 + v\N[v] & \text{if } x = y 
                  \end{cases}
      \end{array}
   \end{equation}
   where in the first step we applied $\omega(\kl{x}) = \kl{x^{-1}}$ (which follows from 
   $\omega = \iota \circ \widebar{(-)}$ and $\iota(\kl{x}) = \kl{x^{-1}}$), in 
   the third step we used $\varepsilon(\std{x} \std{y}) = \delta_{x, y^{-1}}$
   with $\delta_{x, y^{-1}}$ the Kronecker delta and in the last step we plugged
   in $h_{z^{-1}, x^{-1}} = h_{z, x}$ for all $z, x \in W$. Finally, Soergel's conjecture 
   together with the definition of the Kazhdan-Lusztig polynomials 
   implies that $h_{w',  w} \in v \N[v]$ for $w' < w$ and $h_{w,w} = 1$ for all $w, w' \in W$.
\end{proof}

\subsection{Generalized braid groups and the \texorpdfstring{$2$}{2}-braid group}
\label{secGenBraidGroups}

By dropping the condition $s^2 = 1$ for all $s \in S$ in the presentation of $W$, we get 
the presentation of the generalized braid group and the braid monoid corresponding to $(W,S)$:
\begin{align*}
   \br[(W, S)] &= \; \langle s \in S \; \vert \; \underbrace{sts\dots}_{m_{s,t} \text{ terms}} = 
      \underbrace{tst\dots}_{m_{s,t} \text{ terms}} \rangle \text{ in the category of groups;}\\
   \brm[(W, S)] &= \; \langle s \in S \; \vert \; \underbrace{sts\dots}_{m_{s,t} \text{ terms}} = 
      \underbrace{tst\dots}_{m_{s,t} \text{ terms}} \rangle \text{ in the category of monoids.}
\end{align*}
Note that there is a surjective monoid homomorphism $\psi: \brm[(W,S)] \rightarrow W$ which
is the identity on generators. The map $\psi$ admits a set-theoretic section $\varphi: W \rightarrow \brm[(W,S)]$ 
which sends an element in $W$ to the word in the braid monoid corresponding to one of its reduced 
expressions. (This is well defined as any two reduced expressions of an element in $W$ can be related 
using only the braid relations.) The elements in its image are called \emph{reduced braids}. 
Set $\brred[(W, S)] \defeq \im(\varphi)$. In this setting we have the \emph{Garside normal form} 
for $\brm[(W,S)]$ as described in \cite[Section 4]{Mi}:
\begin{thm}[Garside normal form]
   For any positive braid $\sigma \in \brm{(W, S)}$ there exists a unique sequence
   $(w_m, w_{m-1}, \dots, w_1)$ of reduced braids such that $\sigma = w_m w_{m-1} \dots w_1$ and
   for all $1 \leqslant i \leqslant m$ the reduced braid $w_i$ is non-trivial and the unique maximal
   reduced braid that occurs as a right divisor of $w_m w_{m-1} \dots w_i$.
   
   The sequence $(w_m, w_{m-1}, \dots, w_1)$ is called the \emph{Garside normal form} of $\sigma$.
   Each $w_i$ for $1 \leqslant i \leqslant m$ is called a \emph{Garside factor}.
\end{thm}

The following two results will be important for us. The first one shows that being a Garside normal
form can be checked locally (see \cite[Proposition 4.1]{Mi}) and the second results relates a Garside
normal form to the left and right descent sets of its factors in the Coxeter group 
(see \cite[Corollary 4.2]{Mi}):

\begin{prop}
   \label{propGarsideNormFormLocal}
   $(w_m, w_{m-1}, \dots, w_1)$ is a Garside normal form of the element $w_m w_{m-1} \dots w_1$ if and only if
   for all $1 \leqslant i < m$ the sequence $(w_{i+1}, w_i)$ is a Garside normal form of $w_{i+1}w_i$.
\end{prop}

\begin{lem}
   \label{lemGarsideNormFormChar}
   Let $x, y \in W$. $(\varphi(x), \varphi(y))$ is a Garside normal form of $\varphi(x)\varphi(y)$ 
   in the positive braid monoid $\brm[(W, S)]$ if and only if $\desc{R}(x) \subseteq \desc{L}(y)$ 
   in the Coxeter group.
\end{lem}

From the presentations given above, it is immediate that there is a canonical morphism of monoids 
$\brm[(W,S)] \rightarrow \br[(W,S)]$ which is the identity on the generators in $S$. It is a natural question
whether this morphism is injective, i.e. whether $\brm[(W,S)]$ can be identified with the submonoid 
of $\br[(W,S)]$ generated by $S$. This is the case in all types. In finite type it is well known (see
for example \cite[Corollary 3.2]{Mi}) as $\br[(W,S)]$ is the group of fractions of $\brm[(W,S)]$. 
It has been extended to arbitrary type in \cite{Pa} using extensive calculations. In this paper we will
give an alternative proof of this in arbitrary type.

The following result will be helpful to prove the faithfulness of the $2$-braid group in finite type
(see \cite[Lemma 2.3]{BT} for a proof) and relies on the fact that in finite type the Garside
normal from can be extended from $\brm{(W, S)}$ to $\br{(W, S)}$:
\begin{lem}
   \label{lemBrmInjectivity}
   Assume that $(W, S)$ is a Coxeter system of finite type.
   A group homomorphism $\rho: \br[(W, S)] \rightarrow G$ is injective if and only if the induced monoid
   homomorphism $\rho^+: \brm[(W, S)] \hookrightarrow \br[(W, S)] \rightarrow G$ is injective.
\end{lem}

Define the elementary Rouquier complexes corresponding to a simple reflection $s \in S$ as follows
\begin{alignat*}{3}
   F_s & \defeq (&0 \longrightarrow  B_s &\longrightarrow R(1) 
   \longrightarrow 0) 
   \  && \text{with } a \otimes b \mapsto ab  \\
   E_s = F_{s^{-1}} & \defeq (0 \longrightarrow &R(-1) \longrightarrow  B_s &\longrightarrow 0)
   && \text{with } 1 \mapsto \frac{1}{2}(\alpha_s \otimes 1 + 1 
   \otimes \alpha_s )
\end{alignat*}
where in both complexes $B_s$ sits in cohomological degree $0$.

In \cite{Ro} Rouquier showed that in the bounded homotopy category of Soergel bimodules $\Kb$ these complexes 
are inverse to each other and satisfy the braid relations up to canonical isomorphism. We 
define \emph{the $2$-braid group}, denoted by $\brcat$, as the full monoidal subcategory 
of $\Kb$ generated by $F_{s^{-1}}$ and $F_s$ for $s \in S$. Note that the set of isomorphism classes 
of objects in $\brcat$, denoted by $Pic(\brcat)$, forms a group together with the binary operation induced 
by the tensor product. $Pic(\brcat)$ is also called \emph{the Picard group} of the monoidal
category $\brcat$. Rouquier formulates the following conjecture (see \cite[Conjecture 9.8]{Ro}):

\begin{conj}[Faithfulness of the $2$-braid group]
   The natural map $\br[(W, S)] \longrightarrow Pic(\brcat)$ is an isomorphism. 
\end{conj}

This conjectures follows in type $A$ from work by Khovanov and Seidel (see \cite{KS}) and
in simply-laced, finite type from results by Brav and Thomas (see \cite{BT}). In this paper we will prove
the faithfulness of the $2$-braid group in finite type (extending the previous results to
non-simply laced finite type).

\section[Consequences of combinatorial formulas]{Consequences of the multiplication and the \texorpdfstring{$\Hom$}{Hom}-formula}
\label{secHomFormCons}

Recall the left cells $\lcellFixedRightDesc{s}$ for $s \in S$ and the two-sided cell $\cellUniqueRex$ which were introduced at
the end of \Cref{secHeckeAlg}. In this section we want to show that there is particularly nice 
choice of $s \in S$ for the left cell $\lcellFixedRightDesc{s}$, calculate all the Kazhdan-Lusztig polynomials for 
the elements in $\lcellFixedRightDesc{s}$ and draw some conclusions using Soergel's $\Hom$-formula.

First note that the graph $\Gamma_s$ encodes all information necessary for the left 
cell module corresponding to $\lcellFixedRightDesc{s}$:
\begin{lem}
   \label{lemCellMult}
   For $w \in \lcellFixedRightDesc{s}$ and $r \in S$ we have in $\heck( \leqcell{L} \lcellFixedRightDesc{s}) / 
   \heck(\lcell{L} \lcellFixedRightDesc{s})$:
   \[ \kl{r} \kl{w} = \begin{cases}
         (v + v^{-1}) \kl{w} & \text{if } r \in \desc{L}(w) \\
         \sum_{\substack{y \in \lcellFixedRightDesc{s} \text{ s.t. } r \in \desc{L}(y) \\ \text{and } \{y, w\} \in E(\Gamma_s)}} \kl{y} 
            & \text{if } r \nin \desc{L}(w) 
      \end{cases}\]
\end{lem}
\begin{proof}
   The left handed multiplication formula from \cite[Formula 2.3.a and 2.3.c]{KL}) reads 
   in Soergel's normalization as follows:
   \begin{equation}
      \label{eqnLHMultForm}
      \kl{r} \kl{w} =
      \begin{cases}
         (v + v^{-1})\kl{w} & \text{if } rw < w \text{,}\\
         \kl{rw} + \sum_{\substack{y < w \text { s.t.}\\ ry < y}}{\mu(y, w)\kl{y}} 
            & \text{otherwise.}
      \end{cases}
   \end{equation}
   Using \cite[Claim 2.3.e]{KL} one can rewrite this as:
   \begin{equation}
      \kl{r} \kl{w} =
      \begin{cases}
         (v + v^{-1})\kl{w} & \text{if } r \in \desc{L}(w) \text{,}\\
         \sum_{\substack{y \in W \text{ s.t.} \\ r \in \desc{L}(y)}} \mu(y, w) \kl{y} & 
            \text{if } r \notin \desc{L}(w) \text{.}
      \end{cases}
   \end{equation}
   The last formula together with \cref{itmGammaSEdge} from \cref{lemGammaSFacts} gives the claim.
\end{proof}

Similarly, \cref{eqnLHMultForm} implies (observe that in simply-laced type for any vertex of
$\Gamma_s$ the map $\pi_s$ is injective on the set of its neighbours):
\begin{lem}
   \label{lemMultFormConsSL}
   Assume $(W, S)$ to be of simply-laced type. For the unique reduced expression 
   $\underline{w}=s_1 s_2 \dots s_k$ of $w \in \lcellFixedRightDesc{s}$ we have in 
   $\heck( \leqcell{L} \lcellFixedRightDesc{s}) / \heck(\lcell{L} \lcellFixedRightDesc{s})$:
   \[ \kl{w} = \kl{\underline{w}} = \kl{s_1} \kl{s_2} \dots \kl{s_k}\]
\end{lem}
\begin{lem}
   Assume $(W, S)$ to be of non-simply-laced type. Choose $s \in S$ among a pair $\{s, t\} \subseteq S$
   with $m_{s, t} \geqslant 4$. Each element $w \in \lcellFixedRightDesc{s}$ can be uniquely written as
   $w = w_l \dots w_2 w_1$ for some $l \in \N$ where for each $1 \leqslant i \leqslant l$ odd 
   $w_i = \widehat{k_i(q_i, r_i)} = \dots q_i r_i$ (an alternating product of $k_i$ terms) for some $k_i \in \Z$ 
   lies in the standard parabolic subgroup generated by $\{q_i, r_i\} \subseteq S$ with $m_{q_i, r_i} \geqslant 4$
   and for each $2 \leqslant j \leqslant l$ even $w_j$ lies in a standard parabolic subgroup of simply-laced type.
   In $\heck( \leqcell{L} \lcellFixedRightDesc{s}) / \heck(\lcell{L} \lcellFixedRightDesc{s})$ we have 
   \[\kl{w} = \kl{w_l} \dots \kl{w_2} \kl{w_1}\] 
   and the \cref{lemMultFormConsSL} may be applied to all $\kl{w_i}$ with $2 \leqslant i \leqslant l$ even.
   
   If in addition $(W, S)$ is of finite type, then there is a unique pair $\{s, t\} \subseteq S$
   with $m_{s, t} \geqslant 4$ and $l$ in the form above is smaller or equal to $2$.
\end{lem}

From now on, we assume that if $(W, S)$ is of non-simply laced type, then 
$s \in S$ is chosen among a pair $\{s, t\} \subseteq S$ with $m_{s, t} \geqslant 4$ (which is
unique if $W$ is finite).

\begin{lem}
      \label{lemHomsGen}
      For $x \neq y \in \lcellFixedRightDesc{s}$ there are only morphisms of degree one from $B_x$ to $B_y$ 
      if either $\{x, y\}$ is an edge in $\Gamma_s$ or $\desc{L}(x) = \desc{L}(y)$
      and $x$ and $y$ are comparable in the Bruhat order. 
      
      If $x$ and $y$ are connected by an edge in $\Gamma_s$, then a morphism of 
      degree one from $B_x$ to $B_y$ is unique up to scalar.
\end{lem}
\begin{proof}
   Rewriting \cref{eqnHomCalc} we get for the graded rank of 
   $\Hom_{\sbim}^{\bullet}(B_x, B_y)$:
   \begin{align*}
      (\kl{x}, \kl{y}) &= \sum_{z \leqslant x, y} h_{z, x} h_{z, y} \\
                       &\in \mu(x,y)v + v^2 \N[v]
   \end{align*}
   Thus there are morphisms of degree one from $B_x$ to $B_y$ if and only if 
   $\mu(x, y)$ is non-zero. If this is the case, either the left descent
   sets of $x$ and $y$ disagree which by \cref{lemGammaSFacts} (iii) is equivalent 
   to $\{x, y\}$ being an edge in $\Gamma_s$ or they agree which gives the second 
   case of the statement as $\mu(-,-)$ is only non-zero for comparable elements.
   
   The second part of the lemma is just a reformulation of \cref{lemGammaSFacts} (iii).
\end{proof}

\begin{ex}
   We want to give an example showing that the second case in \cref{lemHomsGen} can actually
   occur, i.e. there may exist a degree $1$ morphism between indecomposable Soergel bimodules
   corresponding to $x, y \in \lcellFixedRightDesc{s}$ such that $\desc{L}(x) = \desc{L}(y)$
   and $x$ and $y$ are comparable in the Bruhat order. Consider an affine Coxeter group $W$
   of type $\widetilde{A_2}$ with generators $S = \{s, t, u\}$. A simple calculation shows
   $h_{s, stus} = v^3 + v$. Thus Soergel's $\Hom$-formula implies that there is a degree $1$
   morphism from $B_{stus}$ to $B_s$ and $B_s$ to $B_{stus}$.   
\end{ex}

For the rest of the section we will deal with the finite case in which we can say a little
more.

\begin{lem}
   \label{lemCsSmooth}
   Assume that $(W, S)$ is a Coxeter group of finite type.
   Under the choices made above, all elements in $\lcellFixedRightDesc{s}$ are rationally smooth. In other words,
   for $w \in \lcellFixedRightDesc{s}$ and $y \leqslant w$ we have: $h_{y, w} = v^{l(w) - l(y)}$.
\end{lem}

\begin{proof}
   For all elements lying in the parabolic subgroup generated by $\{s, t\}$, the result is well known 
   (see \cite[Claim 2.1]{E3}) as this is simply a dihedral group. Choose $w \in \lcellFixedRightDesc{s}$ that does not 
   lie in this parabolic subgroup.
   
   By induction, we may assume that we haven proven the statement for all 
   $w' \in \lcellFixedRightDesc{s}$ such that $w'< w$.Let $\pi_s(w) = r \in S$. 
   Set $w' = rw < w$ and choose $y \leqslant w$. Note that $r$ does not occur in 
   $rw$. We obtain an inductive formula for the Kazhdan-Lusztig polynomials from 
   the left-handed multiplication formula in \cref{eqnLHMultForm} by expressing 
   each Kazhdan-Lusztig basis element in terms of the standard basis and by 
   comparing coefficients
   \begin{align*}
      h_{y, w} &= h_{ry, rw} + v^{c_y} h_{y, rw} - \sum_{\substack{y \leqslant z < rw \text{ s.t.}\\ rz < z}}
               \mu(z, rw) h_{y, z} \quad \text{ where }  c_y =
               \begin{cases}
                  1 & \text{if } ry > y \text{,} \\
                  -1 & \text{otherwise.}
               \end{cases} \\
      &= \begin{cases}
            h_{ry, rw} & \text{if } ry < y \text{,} \\
            vh_{y, rw} & \text{if } ry > y \text{.}
         \end{cases} \\
      &= v^{l(w) - l(y)}
   \end{align*}
   where the sum $\sum \mu(z, rw) h_{y, z}$ vanishes as $r$ does not occur in $rw$ and thus there cannot
   be an element $z < rw$ with $r \in \desc{L}(z)$. In addition, we used $y \nleqslant rw$ (resp. $ry \nleqslant rw$)
   if $ry < y$ (resp. if $ry > y$) and thus $h_{y, rw} = 0$ (resp. $h_{ry, rw} = 0$) for basically the same reason.
   In the last step we simply plugged in the Kazhdan-Lusztig polynomials for $w' = rw$ which is by induction
   rationally smooth.
\end{proof}

\begin{remark}
   \begin{itemize} 
      \item Note that the last lemma is no longer true if $s$ is not chosen among the unique pair
            $\{s, t\} \subseteq S$ with $m_{s, t} \geqslant 4$. Consider the example
            given by the following Coxeter graph:
            
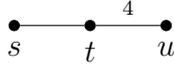
\[ \begin{tikzpicture}[node distance=1cm, auto, baseline=(current  bounding  box.center)]
                  \tikzstyle{vertex}=[circle, fill, inner sep=1.5pt, outer sep=0mm];
                  
                  \node [vertex, label=below:$s$] (v1) at (0,0) {};
                  \node [vertex, label=below:$t$, right of=v1] (v2) {};
                  \node [vertex, label=below:$u$, right of=v2] (v3) {};
                  \draw (v1) to (v2);
                  \draw (v2) to node[label=above:\scriptsize$4$] {} (v3);
               \end{tikzpicture} \]
            For $\Gamma_s$ we get:
            
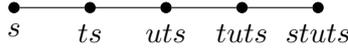
\[ \begin{tikzpicture}[node distance=1cm, auto, baseline=(current  bounding  box.center)]
                  \tikzstyle{vertex}=[circle, fill, inner sep=1.5pt, outer sep=0mm];
                  
                  \node [vertex, label=below:\small$s$] (v1) at (0,0) {};
                  \node [vertex, label=below:\small$ts$, right of=v1] (v2) {};
                  \node [vertex, label=below:\small$uts$, right of=v2] (v3) {};
                  \node [vertex, label=below:\small$tuts$, right of=v3] (v4) {};
                  \node [vertex, label=below:\small$stuts$, right of=v4] (v5) {};
                  \draw (v1) to (v2);
                  \draw (v2) to (v3);
                  \draw (v3) to (v4);
                  \draw (v4) to (v5);
               \end{tikzpicture} \]
           A calculation yields the following Kazhdan-Lusztig polynomials
           \begin{align*}
               h_{us, stuts} &= v^3 + v\\
               h_{u, stuts} &= v^4 + v^2\\
               h_{s, stuts} &= v^4 + v^2\\
               h_{e, stuts} &= v^5 + v^3
           \end{align*}
         which show that $stuts$ is not rationally smooth.
      \item The last lemma holds in slightly more generality: It is true for all Coxeter groups
            whose Coxeter graph is a tree with at most one pair $\{s, t\} \subseteq S$ such 
            that $m_{s, t} \geqslant 4$.
   \end{itemize}
\end{remark}

\begin{lem}
   \label{lemHoms}
   Assume that $(W, S)$ is of finite type.
   Let $x, y \in \lcellFixedRightDesc{s}$. Choose $w \in W$ the unique maximal element such that $w < x$ and $w < y$.
   Then $\Hom_{\sbim}^{\bullet}(B_x, B_y)$ is concentrated in degrees $\geqslant l(x) + l(y) - 2l(w )$
   and is of dimension $1$ in degree $l(x) + l(y) - 2l(w)$. 
   Moreover, the morphisms are concentrated in even (resp. odd) degrees if $l(x) - l(y)$ is even (resp. odd).
   In particular, there are non-zero morphisms of degree one from $B_x$ to $B_y$ if and only if 
   $x$ and $y$ are connected by an edge in $\Gamma_s$.
\end{lem}
\begin{proof}
   First note that $w$ lies in $\lcellFixedRightDesc{s}$ and that it is the unique maximal subexpression shared by
   the unique reduced expressions of $x$ and $y$. (In the simply laced case it is the first vertex 
   the two paths from $x$ to $s$ and from $y$ to $s$ share.) Then plugging the results from
   \cref{lemCsSmooth} into \cref{eqnHomCalc} we get for the graded rank of 
   $\Hom_{\sbim}^{\bullet}(B_x, B_y)$:
   \begin{align*}
      (\kl{x}, \kl{y}) &= \sum_{z \leqslant w} h_{z, x} h_{z, y} \\
                       &= \sum_{z \leqslant w} v^{l(x) + l(y) - 2l(z)}
   \end{align*}
   
   
   The lowest degree term is $v^{l(x) + l(y) - 2l(w)}$. Applying Soergel's $\Hom$-formula  
   from \cref{thmSoergelHom} yields the result. The graded rank shows that depending
   on the parity of $l(x) - l(y)$ all generators are in either even or odd degree. Due to our
   choice of grading on $R$, this implies that all morphisms are concentrated in even (resp. odd)
   degrees if $l(x) - l(y)$ is even (resp. odd).
\end{proof}

\begin{remark}
   It should be noted that for $x,y \in W$ the parity vanishing of $\Hom_{\sbim}^{\bullet}(B_x, B_y)$
   is a more general fact that holds in any Coxeter system.
\end{remark}


\section{Categorified left cell modules}
\label{secCellModCat}

The goal of this section is to construct a categorification of the left cell module corresponding 
to $\lcellFixedRightDesc{s}$ by mimicking the definition of the left cell module on the categorical 
level. In order to do so, we will need some results that show that the Grothendieck
group behaves well with respect to suitable quotients and subcategories.

\begin{lem}
   \label{lemQuotKS}
   Let $\cat{C}$ be an essentially small Krull-Schmidt category and $X$ a subclass of the 
   indecomposable objects, closed under isomorphism. Let $\ideal{J}$ be the $2$-sided ideal of 
   all morphisms factoring through a finite direct sum of objects in $X$. Then the following holds:
   \begin{enumerate}
      \item $\cat{C} / \ideal{J}$ is a Krull-Schmidt category.
      \item Let $B$ be a set of representatives of all isoclasses of indecomposable objects in $\cat{C}$.
            A $\Z$-basis of $\Galg{\cat{C} / \ideal{J}}$ is given by $B \setminus X$.
      \item The functor $\cat{C} \rightarrow \cat{C} / \ideal{J}$ is additive and induces the following isomorphism of 
            abelian groups on the level of Grothendieck groups: 
            \[ \G{C} \; / \langle [M] \; \vert \; M \in X \rangle \quad \overset{\sim}{\longrightarrow} 
            \quad \Galg{\cat{C} / \ideal{J}} \]
   \end{enumerate}
   
   Assume in addition that $\cat{C}$ is graded, monoidal and that $X$ is closed under grading shifts.
   Furthermore, assume that no indecomposable object $Y$ is isomorphic to one of its grading shifts $Y(m)$
   for $m \in \Z$. Let $B$ be a set of representatives of all isoclasses of indecomposable objects 
   in $\cat{C}$ up to grading shift. Assume that $B \cap X$ is the union of all two-sided 
   cells belonging to a lower set in the two-sided cell preorder (see \cref{secHeckeAlg})
   of $\G{C}$ with respect to the $\Z[v, v^{-1}]$-basis $B$. Then the following holds:
   
   \begin{enumerate}
      \item $\cat{C} / \ideal{J}$ is a graded, monoidal category.
      \item A $\Z[v, v^{-1}]$-basis of $\Galg{\cat{C} / \ideal{J}}$ is given by $B \setminus X$.
      \item The functor $\cat{C} \rightarrow \cat{C} / \ideal{J}$ is a strict monoidal functor and 
            the isomorphism above is an isomorphism of $\Z[v, v^{-1}]$-algebras.
   \end{enumerate}
\end{lem}

\begin{lem}
   \label{lemSubcatKS}
   Let $\cat{C}$ be an essentially small (graded) Krull-Schmidt category 
   such that no indecomposable object $Y$ is isomorphic to one of its grading 
   shifts $Y(m)$ for $m \in \Z$. Let $B$ be a subset of the representatives of 
   all isoclasses of indecomposable objects in $\cat{C}$ (up to grading shift). 
   Consider the full additive (graded) subcategory $\cat{C}_B$ of $\cat{C}$ 
   generated by all objects isomorphic to an object in $B$, also denoted 
   by $\langle M \; \vert \; M \in B \rangle_{\oplus}$ (resp. $\langle M \; 
   \vert \; M \in B \rangle_{\oplus, (-)}$ ). Then the following holds:
      \begin{enumerate}
      \item 
            $\cat{C}_B$ is a (graded) Krull-Schmidt category.
      \item $B$ is a $\Z$- (resp. $\Z[v, v^{-1}]$-) basis of $\Galg{\cat{C}_B}$.
      \item 
            The functor $\cat{C}_B \hookrightarrow \cat{C}$ induces the following isomorphism of 
            $\Z$- (resp. $\Z[v, v^{-1}]$-) modules on the level of Grothendieck-groups:
            \[ \Galg{\cat{C}_B} \quad \longrightarrow \quad \langle [M] \; \vert \; M \in B \rangle \; \subset \; \G{C}\]
   \end{enumerate}
\end{lem}

Consider in $W$ the set $\cellNonUniqueRex$ of all elements that do not admit a unique reduced expression. Note 
that $\cellNonUniqueRex$ is the union of all two-sided cells belonging to a lower set in the two-sided cell 
preorder of $\heck$ with respect to the Kazhdan-Lusztig basis. The Hasse diagram of the 
two-sided cell preorder in $\heck$ can be pictured as follows:
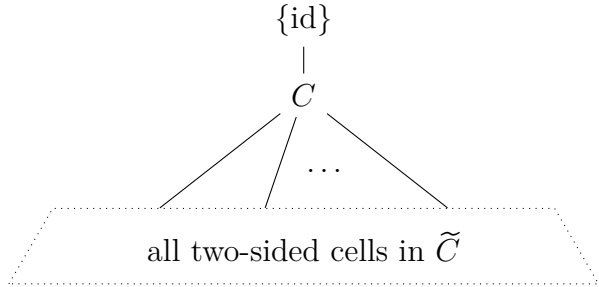
\begin{center}
   \begin{tikzpicture}[scale=0.75]
      \node (id) at (0,0) {$\{ \id \}$};
      \node[below of=id] (C) {$\cellUniqueRex$};
      \node[draw, dotted, below of=C, node distance=2cm, trapezium, trapezium left angle=60, 
            trapezium right angle=60, minimum height=1cm] (C') {all two-sided cells in $\cellNonUniqueRex$};
      \node[below of=C, xshift=0.3cm] {$\dots$};

      \draw (id) to (C);
      \draw (C) to (C'.165);
      \draw (C) to (C'.135);
      \draw (C) to (C'.15);
   \end{tikzpicture}
\end{center}
In order to see this, use the characterization of the two-sided cell preorder in terms of left
and right descent sets. Let $\ideal{J}$ be the two-sided ideal of all morphisms in $\sbim$ factoring through
a finite direct sum of objects in the class $\{ M \obj{\sbim} \; \vert \; M \cong B_w(k) 
\text{ for some } w \in \cellNonUniqueRex \text{ and } k \in \Z \}$.
\Cref{lemQuotKS} shows that the quotient category $\sbim / \ideal{J}$, which will also
be denoted by $\sbim \; / \langle B_w \; \vert \; w \in \cellNonUniqueRex \rangle_{\oplus, \, (-)}$,
is a graded monoidal Krull-Schmidt category.

Inside this category we will study the full additive graded subcategory $\catFixedRightDesc{s} = 
\langle B_w \; \vert \; w \in \lcellFixedRightDesc{s} \rangle_{\oplus, (-)}$. From \cref{lemSubcatKS} we see 
that $\catFixedRightDesc{s}$ is still a graded Krull-Schmidt category and that 
$\Galg{\catFixedRightDesc{s}}$ is a free $\Z[v, v^{-1}]$-module with $\{ [B_w] \; \vert \; 
w \in \lcellFixedRightDesc{s} \}$ as basis. In addition, $\Galg{\catFixedRightDesc{s}}$ sits inside
$\Galg{\sbim / \langle B_w \; \vert \; w \in \cellNonUniqueRex \rangle_{\oplus, \, (-)}}$ which 
by \cref{lemQuotKS} is isomorphic to $\G{\sbim} / \G{\sbim}(\lcell{LR} \cellUniqueRex)$ as a 
$\Z[v, v^{-1}]$-algebra. 

Note that for any element $y \in \lcellFixedRightDesc{s}$ and $x \in W$ such that $x \leqcell{L} y$, 
the characterization of the left cell preorder in terms of left descent sets implies that $x$ 
either lies in $\widetilde{C}$ or in $\lcellFixedRightDesc{s}$. Therefore the set 
$\{\lcell{L} \lcellFixedRightDesc{s} \}$ is contained in $\{ \lcell{LR} \cellUniqueRex \}$.
Due to $\heck (\lcell{L} \lcellFixedRightDesc{s}) \subseteq \heck(\lcell{LR} \cellUniqueRex)$, the following square can be completed:
\[ \begin{xy} 
   \xymatrix{
      {\heck(\leqcell{L} \lcellFixedRightDesc{s})} \ar@{^(->}[r] \ar@{->>}[d] & {\heck} \ar@{->>}[d] \\
      {\heck(\leqcell{L} \lcellFixedRightDesc{s}) / \heck (\lcell{L} \lcellFixedRightDesc{s})} \ar@{^(..>}[r] & 
         {\heck / \heck(\lcell{LR} \cellUniqueRex)} }
   \end{xy} \]
which shows that the left cell module corresponding to $\lcellFixedRightDesc{s}$ is isomorphic to
the submodule spanned by $\{ \kl{w} \; \vert \; w \in \lcellFixedRightDesc{s} \}$ inside 
$\heck / \heck(\lcell{LR} \cellUniqueRex)$.

After decategorifying all categorical constructions carry over to $\heck$ without
difficulty:
\begin{itemize}
   \item All ideals in $\G{\sbim}$ considered so far are generated by classes of 
         indecomposable Soergel bimodules.
   \item Soergel's categorification theorem states that $\G{\sbim}$ and $\heck$ are isomorphic as $\Z[v, v^{-1}]$-algebras.
   \item Soergel's conjecture implies that the  $\Z[v, v^{-1}]$-basis of $\G{\sbim}$ given by the classes of perverse 
         indecomposable Soergel bimodules is matched with the Kazhdan-Lusztig basis of $\heck$ under this isomorphism.
\end{itemize}
Therefore we get:  
\[ \begin{xy} 
   \xymatrix{
      {\Galg{\catFixedRightDesc{s}}} \ar[r]^-{\cong} & {\langle [B_w] \; \vert \; w \in \lcellFixedRightDesc{s} \rangle}
         \ar@{^(->}[r] \ar@{..>}[d]^-{\cong} & \G{\sbim} / \G{\sbim} ( \lcell{LR} \cellUniqueRex) \ar[d]^-{\cong} \\
      {\heck(\leqcell{L} \lcellFixedRightDesc{s}) / \heck (\lcell{L} \lcellFixedRightDesc{s})}
         \ar[r]^-{\cong} & {\langle \kl{w} \; \vert \; w \in \lcellFixedRightDesc{s} \rangle} \ar@{^(->}[r] & 
         {\heck / \heck(\lcell{LR} \cellUniqueRex)} }
   \end{xy} \]

Note that for any Soergel bimodule $M \in \sbim$ we have an additive endofunctor on $\catFixedRightDesc{s}$ given by
$M \otimes (-)$ because $\lcellFixedRightDesc{s}$ is a left cell in $\G{\sbim}$ and we have killed all indecomposables
indexed by some element $\lcell{L} \lcellFixedRightDesc{s}$ in $\sbim \; / \langle B_w \; \vert \; w \in \cellNonUniqueRex \rangle_{\oplus, \, (-)}$. This endofunctor descends to the $\Z[v, v^{-1}]$-module endomorphism on $\Galg{\catFixedRightDesc{s}}$ given by left multiplication with $[X]$. Therefore we have shown:

\begin{thm}
   \label{thmCat}
   $\Galg{\catFixedRightDesc{s}}$ is isomorphic to the left cell module 
   $\heck(\leqcell{L} \lcellFixedRightDesc{s}) / \heck (\lcell{L} \lcellFixedRightDesc{s})$ 
   corresponding to $\lcellFixedRightDesc{s}$. This isomorphism matches the action of 
   $\G{\sbim}$ on $\Galg{\catFixedRightDesc{s}}$ and the action of $\heck$ on 
   $\heck(\leqcell{L} \lcellFixedRightDesc{s}) / \heck (\lcell{L} \lcellFixedRightDesc{s})$.
\end{thm}

Next, we will define an action of $\brcat$ on $K^b(\catFixedRightDesc{s})$. Since we killed all
indecomposable Soergel bimodules indexed by elements in $\{ \lcell{L} \lcellFixedRightDesc{s}\}$
(and their grading shifts) in $\catFixedRightDesc{s}$, $K^b(\catFixedRightDesc{s})$ can be viewed as a module category
over the monoidal category $K^b(\sbim)$ (i.e. there exists a bifunctor $K^b(\sbim) \times K^b(\catFixedRightDesc{s})
\rightarrow K^b(\catFixedRightDesc{s})$ induced by the tensor product of the corresponding
complexes together with an associator and a left unitor satisfying the usual coherence conditions).
Restricting this action of $K^b(\sbim)$ on $K^b(\catFixedRightDesc{s})$ to $\brcat$ yields
our action. In particular, the action of $\br[(W, S)]$ on $K^b(\sbim)$ descends to $K^b(\catFixedRightDesc{s})$.
More explicitly, for $\sigma \in \br[(W, S)]$ we get an autoequivalence of $K^b(\catFixedRightDesc{s})$ via 
$F_{\sigma} \otimes (-)$. This yields a group homomorphism $\br[(W,S)] \rightarrow 
Iso(Aut(K^b(\catFixedRightDesc{s})))$, where $Iso(Aut(K^b(\catFixedRightDesc{s})))$ are 
the isomorphism classes of autoequivalences on $\catFixedRightDesc{s}$. We will show 
that this group homomorphism is faithful in finite type, even though its decategorification
is not faithful in general (see \cref{exDecatNonFaithful}).

\section{The perverse Filtration}
\label{secPervDef}

In this section we introduce the perverse t-structure on the homotopy category of Soergel bimodules
as described by Elias and Williamson in \cite[Remark 6.2]{EW2}. The main idea is to play out the 
two gradings on a complex of Soergel bimodules, namely the cohomological and the internal grading,
against each other and to consider linear complexes (see \cite{MaO, Ma1, Ma2, Ma3}).

Let $\cat{C}$ be the category of Soergel bimodules or a catorified left cell module
$\catFixedRightDesc{s}$ for some $s \in S$ as introduced in \cref{secCellModCat}.
We try to use the following convention throughout: $n$ denotes the cohomological degree,
$m$ the grading shift and $k$ possible multiplicites.

\begin{defn}
   \label{defPervBimod}
   Let $B \in \cat{C}$ be a Soergel bimodule. $B$ is called \emph{perverse} if 
   $B$ is a direct sum of copies of $B_w$ for $w \in W$ without grading shifts. 
   Fix a choice of decomposition of $B$ into indecomposable Soergel bimodules: 
   \[ B = \bigoplus_{\substack{w \in W \\ m \in \Z}}{B_w^{\oplus k_{w,m}}(m)} \]
	For $j \in \Z$ define the perverse filtration $\dots \subseteq \tauleq{j-1}B 
   \subseteq \tauleq{j}B \subseteq \dots$ of $B$ as follows:
	\[ \tauleq{j}B \defeq \bigoplus_{\substack{w \in W \\ m \geqslant -j}} 
   {B_w^{\oplus k_{w,m}}(m)} \text{.} \]
   Set $\taul{j} \defeq \tauleq{j-1}$ and define $\taugeq{j} B = B/\taul{j}B$.
   Similarly set $\taug{j} \defeq \taugeq{j+1}$. Define the $j$-th perverse cohomology as:
   \[ \pH{j}(B) \defeq (\tauleq{j}(B)/\taul{j}(B))(j) \text{.} \]
\end{defn}

By definition the perverse filtration of a Soergel bimodule always splits 
and subquotients of the perverse filtration are isomorphic to shifted perverse
Soergel bimodules. More explicitly $\tauleq{j} B / \taul{j}B$ is isomorphic 
to $C(-j)$ for  some perverse Soergel bimodule $C \in \cat{C}$ and contains
exactly all those indecomposable summands $B_w(-j)$ of $B$ for $w \in W$. The perverse 
cohomology shifts the subquotients back in order for them to be perverse. Thus 
in this case $\pH{j}(B)$ exactly gives the perverse Soergel bimodule $C$.

\begin{lem}
   Let $j \in \Z$. $\tauleq{j}(-)$ and $\taugeq{j}(-)$ give well-defined additive 
   endofunctors on the category of Soergel bimodules. $\pH{j}(-)$ defines an
   additive functor from the category of Soergel bimodules to the full additive subcategory
   of perverse Soergel bimodules.
\end{lem}
\begin{proof}
   This is an easy consequence of the fact that $\dim \Hom_{\cat{C}}(B_x, B_y) = \delta_{x, y}$ 
   for $x, y \in W$ and that the morphisms of all degrees between two indecomposable 
   Soergel bimodules are concentrated in non-negative degrees.
\end{proof}

The following few results on homotopy minimal complexes hold in any Krull-Schmidt 
category. For concreteness, we will state them for the category of Soergel bimodules.

\begin{defn}
   A complex $F \in C^b(\cat{C})$ is called \emph{minimal} if $F$ does not 
   contain a contractible direct summand.
\end{defn}

Since we use right superscripts to indicate homogeneous components of a graded module, 
we will use left superscripts to indicate the cohomological degree whenever we work 
with cochain complexes of graded modules.

It is easy to see that a complex of the form $\dots \rightarrow 0 \rightarrow X 
\overset{\phi}{\rightarrow} Y \rightarrow 0 \rightarrow \dots$ where $\phi$ is an isomorphism 
in $\cat{C}$ is contractible. The following result is due to Bar-Natan (see \cite{BN2}):

\begin{lemlab}[Gaussian elimination]
   \label{corGaussElim}
   Given a complex $F \in C^b(\cat{C})$ which looks in cohomological degrees $n$ and $n+1$ as follows:
   \[\begin{xy}
      \xymatrix{
         {\dots} \ar[r] & \prescript{n-1}{}{F} \ar[r]^-{d^{n-1}} & M \oplus B \ar[r]^-{\left( 
            \begin{smallmatrix} \alpha & \beta \\ \gamma & \delta \end{smallmatrix} \right)}
            &  M' \oplus B' \ar[r]^-{d^{n+1}} & \prescript{n+2}{}{F} \ar[r] & {\dots} }
      \end{xy} \]
   where $\delta: B \rightarrow B'$ is an isomorphism, $F$ is homotopy equivalent to a 
   complex $F' \in C^B(\cat{C})$ which agrees with $F$ outside the cohomological degrees $n$ and $n+1$
   and looks in these two degrees like:
   \[\begin{xy}
      \xymatrix{
         {\dots} \ar[r] & \prescript{n-1}{}{F} \ar[rr]^-{pr_M \circ d^{n-1}} && M 
            \ar[rr]^-{\alpha - \beta \circ \delta^{-1} \circ \gamma}
            &&  M' \ar[rr]^-{d^{n+1} \circ \iota_{M'}} && \prescript{n+2}{}{F} \ar[r] & {\dots} }
      \end{xy} \text{.} \]
   Moreover, $F'$ is a direct summand of $F$. We will call the passage 
   from $F$ to $F'$ a \emph{Gaussian elimination with respect to $\delta$}.
\end{lemlab}

Given any complex $F \in C^b(\cat{C})$ one can successively eliminate contractible direct
summands to obtain a direct summand $F_{\text{min}} \overset{\oplus}{\subseteq} F$ such that
$F_{\text{min}}$ is minimal. In particular, $F$ is isomorphic to $F_{\text{min}}$ in
$K^b(\cat{C})$ as $F_{\text{min}}$ is homotopy equivalent to $F$ via the inclusion and projection.

\todo[inline]{Add references for the next two results!?}

\begin{lem}
   For a complex $F \in C^b(\cat{C})$ any two minimal complexes $F_1, F_2 
   \overset{\oplus}{\subseteq} F$ are isomorphic in $C^b(\cat{C})$.
\end{lem}

\begin{cor}
   \label{corHomEquiMinComp}
   Homotopy equivalent minimal complexes in $C^b(\cat{C})$ are isomorphic in $C^b(\cat{C})$.
\end{cor}

The filtration on $\cat{C}$ from \Cref{defPervBimod} induces a split diagonal 
filtration on the full subcategory $C^b(\cat{C})_{\text{min}}$ of minimal 
complexes in $C^b(\cat{C})$ as follows:

\begin{defn}
	Let $F$ be a minimal, bounded complex of Soergel bimodules. 
	For $j \in \Z$ define the \emph{perverse filtration} $\dots \subseteq \tauleq{j-1}F 
   \subseteq \tauleq{j}F \subseteq \dots$ of $F$ as follows:
	\[ \prescript{n}{}{(\tauleq{j}F)} \defeq \tauleq{j-n}(\prescript{n}{}{F}) \]
   Define $\taul{j}$, $\taugeq{j}$ and $\taug{j}$ as above. Define the $j$-th perverse 
   cohomology as:
   \[ \pH{j}(F) \defeq (\tauleq{j}(F)/\taul{j}(F))[j] \]
\end{defn}

Note that $\tauleq{j}F$ is a well-defined subcomplex of $F$ because Soergel's $\Hom$-formula implies that the
homomorphisms of all degrees from $B_x$ to $B_y$ for $x, y \in W$ are concentrated in non-negative 
degrees and that there can only be homomorphisms of degree $0$ if $x=y$ and in that 
case every non-zero homomorphism is invertible. Thus non-zero degree $0$ homomorphisms
cannot occur as components in minimal complexes and for every non-zero component
$B_x(m_x) \rightarrow B_y (m_y)$ of the differential in $F$ we have $m_y > m_x$.
In addition, for any minimal complex $F$ of Soergel bimodules as above, there is 
a level-wise split short exact sequence in $C^b(\cat{C})$ for all $j \in\Z$
\[ 0 \longrightarrow \tauleq{j}(F)  \longrightarrow F \longrightarrow \taug{j}(F) \longrightarrow 0 \]
which induces a distinguished triangle in the homotopy category $K^b(\cat{C})$.

The main idea of the next definition is to extend our previous definitions to $K^b(\cat{C})$ via the
equivalence of categories induced by $C^b(\cat{C})_{\text{min}} \hookrightarrow C^b(\cat{C})$
on the level of homotopy categories.

\begin{defn}
   \label{defPervTStruct}
	Let $\KCgeq{0}$ be the full subcategory of $K^b(\cat{C})$ consisting of all complexes which
   are isomorphic to a minimal complex $F \in K^b(\cat{C})$ such that $\taul{0}(F)$ vanishes.
   
   Let $\KCleq{0}$ be the full subcategory of $K^b(\cat{C})$ consisting of all complexes which
   are isomorphic to a minimal complex $F \in K^b(\cat{C})$ such that $\taug{0}(F)$ vanishes.
   
   A complex $F \in K^b(\cat{C})$ is called \emph{linear} or \emph{perverse} if it lies
   in $\KCgeq{0} \cap \KCleq{0}$. A \emph{perverse shift} is a simultaneous shift of the form
   $(-k)[k]$ for some $k \in \Z$.
\end{defn}

\begin{prop}
   The following statements are equivalent for a complex $F \in K^b(\cat{C})$:
   \begin{enumerate}
      \item $F \in \KCgeq{0}$.
      \item Any minimal complex $F' \in C^b(\cat{C})$ which is isomorphic to $F$ in $K^b(\cat{C})$
            satisfies $\taul{0}(F') = 0$.
   \end{enumerate}
\end{prop}
\begin{proof}
   ii) obviously implies i). Since homotopy equivalent minimal complexes in $C^b(\cat{C})$ are isomorphic 
   in $C^b(\cat{C})$ by \cref{corHomEquiMinComp} and the Krull-Schmidt theorem holds in $\cat{C}$, we see that the converse also holds. 
   See \cite[Proposition 3.2.1]{Kr} for the proof that the Krull-Schmidt theorem holds in any 
   Krull-Schmidt category.
\end{proof}

Since this definition is a little technical, let us try to explain it a little: For a minimal complex in $\KCgeq{0}$
an indecomposable Soergel bimodule occuring in cohomological degree $n$ is of the form $B_w (m)$ 
for $m \leqslant n$. Loosely speaking, in $\KCgeq{0}$ the possible shifts of the occuring 
indecomposable are bounded above by the cohomological degree in which the module occurs. For $\KCleq{0}$
replace ``above'' by ``below'' in the last slogan.

So, how should we visualize this definition? For a minimal complex one can keep
track of the indecomposable bimodules occuring in the complex in the form of a table, where 
the columns denote the cohomological degree and the rows correspond 
to the grading shift. We will choose the following convention that
the columns (resp. rows) are labelled by integers in increasing order 
from left to right (resp. from top to bottom). Thus we get a table of the
following form \\
\begin{center}
   \begin{tabular}{ l || c | c | c | c | c }
       & -2 & -1 & 0 & 1 & 2  \\ \hline \hline
    -2 & \cellcolor{black!25} & & & & \\ \hline
    -1 & & \cellcolor{black!25} & & &  \\ \hline
     0 & & & \cellcolor{black!25} & &  \\ \hline
     1 & & & & \cellcolor{black!25} &  \\ \hline
     2 & & & & & \cellcolor{black!25} 
   \end{tabular}
\end{center}
where an indecomposable Soergel bimodule $B_w(m)$ occuring in cohomological degree $n$
appears in the column labelled with $n$ and the row labelled with $m$.
Therefore a minimal complex $F \in \KCgeq{0}$ (resp. $F \in \KCleq{0}$) has only entries in 
cells on or above (resp. below) the marked grey diagonal. The entries in such a table
for a minimal perverse complex are restricted to the grey diagonal.

\begin{lem}
   $(\KCleq{0}, \KCgeq{0})$ gives a non-degenerate t-structure on $K^b(\cat{C})$, 
	the \emph{perverse t-structure} on the homotopy category of Soergel bimodules.
\end{lem}
\begin{proof}
   By definition $\KCgeq{0}$ and $\KCleq{0}$ are full subcategories. They are both proper because we have for example
   $F_s[1] \nin \KCgeq{0}$ and $F_s[-1] \nin \KCleq{0}$.
   
   Set $\KCleq{n} \defeq \KCleq{0}[-n]$ and $\KCgeq{n} \defeq \KCgeq{0}[-n]$. Recall that an object $F \in K^b(\cat{C})$
   lies in $\KCleq{0}$ if and only if for all $n \in \Z$ the shifts of the indecomposables occuring in cohomological
   degree $n$ of a minimal complex isomorphic to $F$ are bounded below by $n$. Therefore we see that 
   a complex $F \in K^b(\cat{C})$ lies in $\KCleq{1}$ if and only if for all $n \in \Z$ the shifts of the 
   indecomposables occuring in cohomological degree $n$ of a minimal complex isomorphic to $F$ are 
   bounded below by $n-1$. This implies $\KCleq{0} \subseteq \KCleq{1}$.
   
   Next, we want to show the $\Hom$-vanishing condition. Let $X, Y \in K^b(\cat{C})$ be minimal complexes such
   that $X$ lies in $\KCleq{0}$ and $Y$ can be found in $\KCgeq{1}$. Since
   $\Hom_{K^b(\cat{C})}(X, Y)$ is a quotient of $\Hom_{C^b(\cat{C})}(X, Y)$ and a morphism $f$ in $\Hom_{C^b(\cat{C})}(X, Y)$
   is a family of morphisms $\{f_n\}_{n \in \Z}$ such that $f_n \in 
   \Hom_{\cat{C}}(\prescript{n}{}{X}, \prescript{n}{}{Y})$ and $d_n^Y \circ f_n = f_{n+1} \circ d_n^X$,
   it suffices to see that $\Hom_{\cat{C}}(\prescript{n}{}{X}, \prescript{n}{}{Y})$ vanishes for all
   $n \in \Z$. This follows again from Soergel's $\Hom$-formula (see \cref{corHomNonNegDeg}) as for $n \in \Z$
   the shifts of the indecomposables occuring in $\prescript{n}{}{X}$ are bounded below by $n$ and
   for the indecomposables in $\prescript{n}{}{Y}$ they are bounded above by $n-1$.
   
   Since we have already noted the existence of the distinguished triangle of the required form prior to
   \cref{defPervTStruct}, it remains to show the non-degeneracy of the perverse t-structure.
   This follows right away by considering a minimal complex and using the classification of the
   indecomposable Soergel bimodules from \cref{thmIndecompSBimClass}.
\end{proof}

The following result implies that after applying $F_s$ to a complex $F \in \KCgeq{0}$, 
the negative perverse cohomology groups of $F F_s$ remain zero (i.e. 
$\pH{k}(F F_s) = 0$ for all $k < 0$). It can be found in \cite[Lemma 6.6]{EW1}. 

\begin{lem}
   \label{lemFsLeftTExact}
   $F_s$ (resp. $E_s$) viewed as an additive endofunctor on $K^b(\cat{C})$ is left (resp. right) t-exact. In other words,
   if $F \in \KCgeq{0}$ then $F_sF \in \KCgeq{0}$.
\end{lem}

The left-handed multiplication formula from \cref{eqnLHMultForm} from the proof
of \cref{lemCellMult} immediately yields the following result:

\begin{cor}
   \label{lemFsLeqComp}
   Let $s \in S$ and $F \in \KCleq{n}$ for some $n \in \Z$. Then $F_sF$ lies in $\KCleq{n+1}$.
\end{cor}

\section{Main results}
\label{secProof}

For the rest of the section, assume that $(W, S)$ is an irreducible Coxeter group 
of arbitrary type. Fix $s \in S$ as before. First we want to check that the action 
of $\br[(W, S)]$ on $K^b(\catFixedRightDesc{s})$ behaves as expected:


\begin{lem}
   \label{lemFirstStep}
   In $K^b(\catFixedRightDesc{s})$ we have for $w \in \lcellFixedRightDesc{s}$ and $r \in S$:
   \[ F_r(B_w) \cong
      \begin{cases}
         (B_w(-1) \rightarrow 0 ) & \text{if } \pi_s(w) = r \text{,} \\
         ( \bigoplus_{\substack{x \in \lcellFixedRightDesc{s} \text{ s.t. } \\ rx < x \text{ and} \\ 
            \{x, w\} \in E(\Gamma_s)}}B_x 
            \rightarrow B_w(1) ) & \text{if } \pi_s(w) \neq r \text{.}
      \end{cases} \]
    where in both complexes the term on the left sits in cohomological degree $0$ and all maps between non-zero
    indecomposable Soergel bimodules are non-zero.
		
    In particular, we have $F_r(B_w) \cong (0 \rightarrow B_w(1))$ if $\pi_s(w) \neq r$ and there is no edge
    between the vertices in $\pi_s^{-1}(r)$ and $w$ in $\Gamma_s$.
\end{lem}
\begin{proof}
   By definition of $F_r$ we have $F_r(B_w) = (B_r B_w \rightarrow B_w(1))$. Apply \cref{lemCellMult} to
   decompose $B_rB_w$ into indecomposable Soergel bimodules and note that all induced components of
   the differential are non-zero as $F_r(B_w)$ is indecomposable ($F_r$ is an autoequivalence and $B_w$ is
   indecomposable). In the case $\pi_s(w) = r$, apply a Gaussian elimination to the direct summand 
   $(B_w(1) \overset{\cong}{\rightarrow} B_w(1))$ of $F_r(B_w)$ to obtain the result.
\end{proof}

Comparing the results from \cref{lemFirstStep} with the formula for $\std{r}\kl{w}$ obtained from \cref{lemCellMult} yields
the following result:

\begin{cor}
   $K^b(\catFixedRightDesc{s})$ gives a categorification of $\heck(\leqcell{L} \lcellFixedRightDesc{s}) / 
   \heck (\lcell{L} \lcellFixedRightDesc{s})$ viewed as $\br[(W, S)]$-module via the group homomorphism
   $\Psi: \br[(W, S)] \rightarrow \heck$, $r \mapsto \std{r}$. In other words, for $\sigma \in \br[(W, S)]$ and 
   $w \in \lcellFixedRightDesc{s}$ the class $[F_{\sigma}(B_w)]$ coincides with $\Psi(\sigma) . \kl{w}$
   in $\heck(\leqcell{L} \lcellFixedRightDesc{s}) / \heck (\lcell{L} \lcellFixedRightDesc{s})$
   under the isomorphism from \cref{thmCat}. 
\end{cor}

\begin{ex}
   \label{exDecatNonFaithful}
   In type $A_n$ the resulting representation of $\br[n+1]$ on the free $\Z[v, v^{-1}]$-module 
   $\Galg{\catFixedRightDesc{s}}$ gives a twisted, reduced Burau representation. 
   Denote by $\sigma_i$ the braid that passes the $i$-th strand under the $(i+1)$-st strand 
   and leaves the other strands untouched and let $S = \{ s_i = (i, i+1) \; \vert \; 
   1 \leqslant i \leqslant n\}$ be the set of simple transpositions. By \cref{lemGammaSFacts} 
   the graph $\Gamma_s$ is isomorphic to $A_n$ for any $s \in S$. Denote by $[i]$ the unique
   element in $\lcellFixedRightDesc{s}$ with $\pi_s([i]) = s_i$.     
   
   As in \cite{KS}, one may check that letting $\sigma_i \in \br[n+1]$ act via 
   the operator on $\Galg{\catFixedRightDesc{s}}$ induced by $F_{s_i}[-1](1)$ defines a representation
   isomorphic to the reduced Burau representation. More precisely, the action of $\sigma_i$
   in the $\Z[v, v^{-1}]$-basis
   \[(v[B_{[1]}], \dots, (-1)^{i-1} v^i [B_{[i]}], \dots, (-1)^n v^n [B_{[n]}])\]
   of $\Galg{\catFixedRightDesc{s}}$ is given by the matrix:
   \[\begin{pmatrix}
   1 & & & & & & \multicolumn{3}{c}{} \\
   & \ddots & & & & & \multicolumn{3}{c}{} \\
   & & 1 & & & & \multicolumn{3}{c}{\raisebox{3ex}[0pt]{\BigZero}} \\
   & & & \tikzmark{left}{1} & 0 & 0 & & & \\
   & & & v^{-2} & -v^{-2} & 1 & & & \\
   & & & 0 & 0 & \tikzmark{right}{1} & & & \\
   \multicolumn{3}{c}{} & & & & 1 & & \\
   \multicolumn{3}{c}{} & & & & & \ddots & \\
   \multicolumn{3}{c}{\raisebox{2ex}[0pt]{\BigZero}} & & & & & & 1
   \end{pmatrix}
   \]
   \DrawBox
   which we recognize as the transpose of the image of $\sigma_i$ under the reduced 
   Burau representation of $\br[n+1]$ (after substituting $t$ for $v^{-2}$) as introduced
   in \cite[chapter 3.3]{KT}.
   
   In \cite{Bi} Bigelow shows that the Burau representation of \br[n] is not 
   faithful for $n \geqslant 5$. Thus it is remarkable that its categorification is 
   faithful (see \cref{thmActFaithful}).
\end{ex}

\begin{defn}
   Let $F \in K^b(\catFixedRightDesc{s})$ be a perverse complex and $w \in W$.
   An indecomposable Soergel bimodule $B_w$ is called an \emph{anchor of $F$} if 
   \[\Hom_{K^b(\catFixedRightDesc{s})}(F, B_w(m)[-m])\] is non-zero for some $m \in \Z$.
   We say that there is an anchor of $F$ corresponding to $t \in S$ if 
   $\Hom_{K^b(\catFixedRightDesc{s})}(F, \bigoplus_{x\in \pi_s^{-1}(t)}B_x(m)[-m])$
   is non-zero for some $m \in \Z$.
\end{defn}

We extend this important definition to non-perverse complexes as follows:

\begin{defn}
   Let $0 \neq F \in K^b(\catFixedRightDesc{s})$ and $k \in \Z$ be maximal such that
   $\pH{k}(F)$ is non-zero. For $w \in W$ we say that $B_w$ is an \emph{anchor of $F$}
   if $\Hom_{K^b(\catFixedRightDesc{s})}(F, B_w(m - k)[-m])$ is non-zero for some
   $m \in \Z$.
\end{defn}

The following result shows that being the anchor of a complex is equivalent
to being the anchor of its highest non-zero perverse cohomology group.

\begin{lem}
   \label{lemAnchorEquiv}
   Let $0 \neq F \in K^b(\catFixedRightDesc{s})$ and $k \in \Z$ be maximal such that
   $\pH{k}(F)$ is non-zero. Then the following statements are equivalent:
   \begin{enumerate}
      \item $B_w$ is an anchor of $F$.
      \item $B_w$ is an anchor of $\pH{k}(F)$.
   \end{enumerate}
\end{lem}
\begin{proof}
   Consider the following distinguished triangle in $K^b(\catFixedRightDesc{s})$
   \[ \tauleq{k-1} F \longrightarrow F \longrightarrow \taugeq{k} F \overset{[1]}{\longrightarrow} \]
   where the term on the right hand side is by assumption isomorphic to
   $\pH{k}(F)[-k]$. Applying $\Hom_{K^b(\catFixedRightDesc{s})}( - , B_w(m-k)[-m])$ and using the
	$\Hom$-vanishing condition in the long exact sequence we obtain an isomorphism:
	\[ \Hom_{K^b(\catFixedRightDesc{s})}(\pH{k}(F)[-k], B_w(m-k)[-m]) \cong 
			\Hom_{K^b(\catFixedRightDesc{s})}(F, B_w(m-k)[-m]) \]
	
\end{proof}

\begin{prop}
   \label{propCharAnchor}
   Let $F \in K^b(\catFixedRightDesc{s})$ be a perverse complex. For $t \in S$ 
   the following statements are equivalent:
   \begin{enumerate}
      \item $B_{w}$ is an anchor of $F$ for some $w \in \pi_s^{-1}(t) \subseteq \catFixedRightDesc{s}$.
      \item There exists $m \in \Z$, $w \in \pi_s^{-1}(t)$ and a minimal complex $F'$ isomorphic 
            to $F$ in $K^b(\catFixedRightDesc{s})$ such that $B_{w}(m)$ occurs 
            as a direct summand in $\prescript{m}{}{F'}$ and all components of the differential 
            ending in this copy of $B_{w}(m)$ are $0$.
      \item $\pH{1}(F_t F) \neq 0$.
   \end{enumerate}
\end{prop}
\begin{proof}
   \emph{i) $\Rightarrow$ ii):} Without loss of generality, we may assume that $F$ 
   itself is a minimal complex. Choose $m \in \Z$ such that 
   $\Hom_{K^b(\catFixedRightDesc{s})}(F, B_{w}(m)[-m])$
   does not vanish. Decompose $\prescript{m}{}{F} = B_{w}(m)^{\oplus k} \oplus C(m)$
   for some perverse Soergel bimodule $C$ in which $B_{w}$ does not occur. Since $F$
   and $B_{w}(m)[-m]$ both are minimal perverse complexes, we have:
   $\Hom_{K^b(\catFixedRightDesc{s})}(F, B_{w}(m)[-m]) = 
   \Hom_{C^b(\catFixedRightDesc{s})}(F, B_{w}(m)[-m])$.
   By assumption we thus have a map of chain complexes
   of the following form:
   \[\begin{xy}
      \xymatrix{
         {\dots} \ar[r] & {\prescript{m-1}{}{F}} \ar[d] \ar[r]^-{d^{m-1}}
            & B_{w}(m)^{\oplus k} \oplus C(m) \ar[r]^-{d^{m+1}} \ar[d]^-{ \left( 
               \begin{smallmatrix}
                  f_1, & f_2, & \dots, & f_k, & 0
               \end{smallmatrix} \right) }
            & {\prescript{m + 1}{}{F}} \ar[d] \ar[r] 
            & {\dots} \\
         {\dots} \ar[r] & 0 \ar[r] & B_{w}(m) \ar[r] & 0 \ar[r] & {\dots} }
   \end{xy} \]
   where $f_i = r_i \id$ with $r_i \in \R$ for all $1 \leqslant i \leqslant k$ and at least 
   one of them is non-zero. Up to automorphism of $F$ we may assume that $r_1 = 1$ and $r_i = 0$
   for all $2 \leqslant i \leqslant k$. Then the commutativity of the left square in the diagram above
   shows that all components of the differential ending in the first copy of $B_w(m)$ in $\prescript{m}{}{F}$
   vanish.

   \emph{ii) $\Rightarrow$ iii):} Let $F' \in K^b(\catFixedRightDesc{s})$ be a minimal complex
   as in ii). It suffices to show that $\pH{1}(F_t F')$ does not vanish. Let $m\in \Z$ and $w \in \pi_s^{-1}(t)$ be such that a 
   copy of $B_{w}(m)$ occurs as summand in $\prescript{m}{}{F'}$ with no non-zero incoming 
   differential components. When tensoring with $F_t$ this copy gives a summand $B_{w}(m-1)$ in 
   $\prescript{m}{}{F_tF'}$ which cannot be cancelled by any Gaussian eliminations by our assumptions 
   (otherwise there would have to be a summand $B_{w}(m-1)$ in $\prescript{m-1}{}{F_tF'}$ together with an ismorphism of the form
   $B_{w}(m-1) \overset{\inkl}{\longrightarrow} B_t (\prescript{m-1}{}{F'}) \overset{1 \otimes d}{\longrightarrow} 
   B_t (\prescript{m}{}{F'}) \overset{\pr}{\longrightarrow} B_{w}(m-1)$ which is impossible by assumption).
   This shows that $\pH{1}(F_t F')$ does not vanish.
   
   \emph{iii) $\Rightarrow$ i):} The only indecomposable Soergel bimodules $B_x$ occuring up 
   to perverse shift in $\pH{1}(F_t F)$ satisfy $\pi_s(x) = t$. Any indecomposable 
   Soergel bimodule in the lowest non-zero cohomological degree of 
   $\pH{1}(F_t F)$ is an anchor. Since $\pH{1}(F_t F)$ does not vanish, we can find 
   $m \in \Z$, $w \in \pi_s^{-1}(t)$ and a non-zero morphism from $\pH{1}(F_t F)$ to 
   $B_{w}(m)[-m]$ in $K^b(\catFixedRightDesc{s})$. As $\pH{k}(F_t F)$ vanishes for all 
   $k > 1$, this non-zero morphism shows that $\Hom_{K^b(\catFixedRightDesc{s})}(F_t F, 
   B_{w}(m-1)[-m])$ is not zero. From the fact that $E_t$ is right adjoint to $F_t$ it follows that
   $\Hom_{K^b(\catFixedRightDesc{s})}(F, E_t B_{\elFixedLeftDesc{t}}(m-1)[-m])$ does not vanish.
   The dual version of \cref{lemFirstStep} shows that $E_t B_{w}(m-1)[-m]$
   is isomorphic to $B_{w}(m)[-m]$ in $K^b(\catFixedRightDesc{s})$ and thus the claim follows.
\end{proof}

As a consequence of \cref{lemAnchorEquiv} and the equivalence iii) $\Leftrightarrow$ i) 
in the last result, we get:
\todo[inline]{Which results about the perverse filtration do we need for the next two results?!}

\begin{cor}
   \label{corAnchorMaxPH}
   Let $0 \neq F \in K^b(\catFixedRightDesc{s})$ be a complex, $t \in S$ and $k \in \Z$
   maximal such that $\pH{k}(F) \neq 0$. Then the following statements are equivalent:
   \begin{enumerate}
      \item $B_{w}$ is an anchor of $F$ for some $w \in \pi_s^{-1}(t)$.
      \item Let $k' \in \Z$ be maximal such that $\pH{k'}(F_t F) \neq 0$.
            Then $k'$ equals $k + 1$ and all the indecomposable Soergel bimodules occuring
            in $\pH{k+1}(F_t F)$ are indexed by an element in $\pi_s^{-1}(t)$.
   \end{enumerate}
   If this is the case and $(W, S)$ is of finite type, there exist $m^1_w, m^2_w, \dots 
   m^{l_w}_w \in \Z$ for $l_w \geqslant 1$ and $w \in \pi_s^{-1}(t)$ such that 
   \[\pH{k+1}(F_t F) \cong \bigoplus_{w \in \pi_s^{-1}(t)} \bigoplus_{i = 1}^{l_w} 
   B_{w}(m^i_w)[-m^i_w] \neq 0 \]
\end{cor}
\begin{proof}
   Only the second part of the corollary needs an explanation. For this note that
   by construction of $\Gamma_s$ any two $x, y \in \lcellFixedRightDesc{s}$ with 
   $\pi_s(x) = t = \pi_s(y)$ are never neighboured in $\Gamma_s$ and thus 
   \cref{lemHoms} shows that there do not exist any non-zero morphisms of degree 
   one between $B_x$ and $B_y$. It follows that $\pH{k+1}(F_t F)$ decomposes 
   into a direct sum of perversely shifted copies of $B_{w}$ for $w \in \pi_s^{-1}(t)$.
\end{proof}

Finally we get:

\begin{cor}
   \label{corNoAnchorMaxPH}
   Let $0 \neq F \in K^b(\catFixedRightDesc{s})$ be a com
   Let $k \in \Z$ (resp. $\widetilde{k} \in \Z$) be maximal such that 
   $\pH{k}(F) \neq 0$ (resp. $\pH{\widetilde{k}}(F_t F) \neq 0$).
   
   If $B_{w}$ is not an anchor of $F$ for any $w \in \pi_s^{-1}(t)$, then $\widetilde{k} = k$.
\end{cor}

Let $q, r \in S$ with $m = m_{q, r} \geqslant 3$ and $w \in \pi_s^{-1}(r)$.
For $1 \leqslant k \leqslant m$ denote by $\widehat{k}_r = \dots rqr$ 
the alternating word of length $k$ in $q$ and $r$ having $r$ in its right descent
set. Similarly for $\widehat{k}_q$. The induced subgraph of $\Gamma_s$ on $\pi_s^{-1}(\{q, r\})$
has a connected component $\Gamma$ containing $w$ which is a graph of type $A_{m-1}$. 
Let $w \mapsto \widetilde{w} \in \catFixedRightDesc{s}$ be the following graph 
automorphism of $\Gamma$:
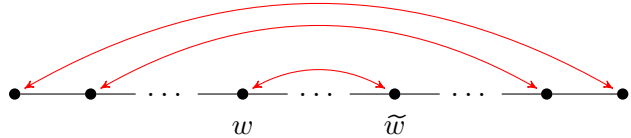
\[ \begin{tikzpicture}[node distance=1cm, auto, baseline=(current  bounding  box.center)]
	\tikzstyle{every label}=[text height=\heightof{(1)}];
	\tikzstyle{vertex}=[circle, fill, inner sep=1.5pt, outer sep=0mm];
	\tikzstyle{autom}=[<->, shorten >= 2pt, shorten <= 2pt, >=stealth', bend left, red];
	
	\node [vertex] (v1) at (0,0) {};
	\node [vertex, right of=v1] (v2) {};
	\node [right of=v2] (vdots1) {$\dots$};
	\node [vertex, label=below:\small$w$, right of=vdots1] (w) {};
	\node [right of=w] (vdots2) {$\dots$};
	\node [vertex, label=below:\small$\widetilde{w}$,right of=vdots2] (w') {};
	\node [right of=w'] (vdots3) {$\dots$};
	\node [vertex, right of=vdots3] (v3) {};
	\node [vertex, right of=v3] (v4) {};
	\draw (v1) to (v2);
	\draw (v2) to (vdots1);
	\draw (vdots1) to (w);
	\draw (w) to (vdots2);
	\draw (vdots2) to (w');
	\draw (w') to (vdots3);
	\draw (vdots3) to (v3);
	\draw (v3) to (v4);
	\draw[autom] (v1) to (v4);
	\draw[autom] (v2) to (v3);
	\draw[autom] (w) to (w');
\end{tikzpicture} \]
Denote by $V(\Gamma)$ the vertex set of $\Gamma$. Note that all elements of 
$V(\Gamma)$ lie in the same left $\langle r, q \rangle$-coset. Let $v \in \lcellFixedRightDesc{s} 
\cup \{1\}$ be the minimal element in this left $\langle r, q \rangle$-coset in $W$.
 Let $t \in \{r, q\}$ be such that 
$tv > v$ and $tv$ admits a unique reduced expression. Then $V(\Gamma)$ is in bijection with
$\{k \in \N \; \vert \; 1 \leqslant k \leqslant m-1\}$ via $k \mapsto [k]  \defeq \widehat{k}_t v \in V$. 
Let $w$ correspond to $k$. Then $\widetilde{w}$ corresponds to $m-k$.
Denote by $d_1 = \min\{k-1, m-1-k\}$ and $d_2 = \max\{k-1, m-1-k\}$ the distances of $w$ from 
the boundaries of $\Gamma$. Applying the graph automorphism swaps $w$ and $\widetilde{w}$ 
and the values of $d_1$ and $d_2$.

\begin{lem}
   \label{lemBraidStep}
   Then we have: \[ F_{\widehat{m-1}_q} B_{w} \cong B_{\widetilde{w}}(1)[-1] \]
	
   In particular, if $m$ is odd, then $F_{\widehat{m-1}_q}$ turns an anchor corresponding
   to an element in $\pi_s^{-1}(r)$ into an anchor in $\pi_s^{-1}(q)$ and if $m$ is even, then
   $F_{\widehat{m-1}_q}$ preserves setwise anchors corresponding to elements in $\pi_s^{-1}(r)$.
\end{lem}

\begin{proof}  
   We will prove the statement by showing the following:
   \begin{claim}
      In $K^b(\catFixedRightDesc{s})$ the complex $F_{\widehat{l}_q} B_w$ is isomorphic to a minimal complex 
      corresponding to a graph in \cref{tabMinComp} where we denoted each indecomposable Soergel bimodule by its index.
      The graphs in \cref{tabMinComp} uniquely determine the isomorphism class of $F_{\widehat{l}_q} B_w$ in 
      $K^b(\catFixedRightDesc{s})$ when placing the indecomposable Soergel bimodules in the left column of 
      each graph without a grading shift in cohomological degree $0$ and using that each arrow stands for 
      a non-zero component of the differential given by the (up to scalar) unique morphism of degree one 
      (see \cref{lemHomsGen}).
   \end{claim}

   \begin{table}[htbp!]
      \begin{adjustbox}{max width=\textwidth}
         \begin{tabular}{cc}
            $\begin{tikzpicture}[node distance=1cm, auto, baseline=(current  bounding  box.center)]
            \tikzstyle{every label}=[text height=\heightof{1}];
            \tikzstyle{every path}=[->, shorten >= 0.5pt, >=stealth']
            \tikzstyle{vertex}=[circle, fill, inner sep=1.5pt, outer sep=0mm];

            \begin{scope}
               \node [vertex, label=left:\footnotesize${[}k+1{]}$] (v1) at (0,0.5) {};
               \node [vertex, label=right:\footnotesize${[}k{]}$] (v0) at (1,0) {};
               \node [vertex, label=left:\footnotesize${[}k-1{]}$, below of=v1] (v-1) {};
               \node [vertex, label=right:\footnotesize${[}k+2{]}$, above of=v0] (v2) {};
               \node [vertex, label=right:\footnotesize${[}k-2{]}$, below of=v0] (v-2) {};
               \node (dotsAbove) at (0.5, 1.75) {\vdots};
               \node (dotsAbove) at (0.5, -1.75) {\vdots};
               \node [vertex, label=left:\footnotesize${[}k+l{]}$, above of=v1, node distance=2.5cm] (vl) {};
               \node [vertex, label=left:\footnotesize${[}k-l{]}$, below of=v-1, node distance=2.5cm] (v-l) {};
               \node [vertex, label=right:\footnotesize${[}k+l-1{]}$, above of=v2, node distance=1.5cm] (vl-1) {};
               \node [vertex, label=right:\footnotesize${[}k-l+1{]}$, below of=v-2, node distance=1.5cm] (v-l+1) {};
               \draw (v1) to (v0);
               \draw (v1) to (v2);
               \draw (v-1) to (v0);
               \draw (v-1) to (v-2);
               \draw (vl) to (vl-1);
               \draw (v-l) to (v-l+1);
               \draw[dotted] (v2)+(157.5:1cm) -- (v2);        
               \draw[dotted] (v-2)+(202.5:1cm) -- (v-2);
               \draw[dotted] (v-l+1)+(157.5:1cm) -- (v-l+1);        
               \draw[dotted] (vl-1)+(202.5:1cm) -- (vl-1);
            \end{scope}
            \end{tikzpicture}$&

            $\begin{tikzpicture}[node distance=1cm, auto, baseline=(current  bounding  box.center)]
            \tikzstyle{every label}=[text height=\heightof{1}];
            \tikzstyle{every path}=[->, shorten >= 0.5pt, >=stealth']
            \tikzstyle{vertex}=[circle, fill, inner sep=1.5pt, outer sep=0mm];

            \begin{scope}
               \node [vertex, label=right:\footnotesize${[}k+1{]}$] (v1) at (1,0.5) {};
               \node [vertex, label=left:\footnotesize${[}k{]}$] (v0) at (0,0) {};
               \node [vertex, label=right:\footnotesize${[}k-1{]}$, below of=v1] (v-1) {};
               \node [vertex, label=left:\footnotesize${[}k+2{]}$, above of=v0] (v2) {};
               \node [vertex, label=left:\footnotesize${[}k-2{]}$, below of=v0] (v-2) {};
               \node (dotsAbove) at (0.5, 1.75) {\vdots};
               \node (dotsAbove) at (0.5, -1.75) {\vdots};
               \node [vertex, label=left:\footnotesize${[}k+l{]}$, above of=v2, node distance=2cm] (vl) {};
               \node [vertex, label=left:\footnotesize${[}k-l{]}$, below of=v-2, node distance=2cm] (v-l) {};
               \node [vertex, label=right:\footnotesize${[}k+l-1{]}$, above of=v1, node distance=2cm] (vl-1) {};
               \node [vertex, label=right:\footnotesize${[}k-l+1{]}$, below of=v-1, node distance=2cm] (v-l+1) {};
               \draw (v0) to (v1);
               \draw (v2) to (v1);
               \draw (v0) to (v-1);
               \draw (v-2) to (v-1);
               \draw (vl) to (vl-1);
               \draw (v-l) to (v-l+1);
               \draw[dotted] (v2) -- ++(22.5:1cm);        
               \draw[dotted] (v-2) -- ++(-22.5:1cm); 
               \draw[dotted] (v-l+1)+(157.5:1cm) -- (v-l+1);       
               \draw[dotted] (vl-1)+(202.5:1cm) -- (vl-1);
            \end{scope}     
            \end{tikzpicture}$ \\

            if $0 \leqslant l \leqslant d_1$  and $l$ is odd & 
            if $0 \leqslant l \leqslant d_1$ and $l$ is even \\[1cm]

            $\begin{tikzpicture}[node distance=1cm, auto, baseline=(current  bounding  box.center)]
            \tikzstyle{every label}=[text height=\heightof{1}];
            \tikzstyle{every path}=[->, shorten >= 0.5pt, >=stealth']
            \tikzstyle{vertex}=[circle, fill, inner sep=1.5pt, outer sep=0mm];

            \node [vertex, label=left:\footnotesize${[}k+l-2{]}$] (vk+l-2) at (0,1) {};
            \node [vertex, label=right:\footnotesize${[}k+l-3{]}$] (vk+l-3) at (1,0.5) {};
            \node [vertex, label=left:\footnotesize${[}k+l{]}$, above of=vk+l-2] (vk+l) {};
            \node [vertex, label=right:\footnotesize${[}k+l-1{]}$, above of=vk+l-3] (vk+l-1) {};
            \node (dots) at (0.5, -0.25) {\vdots};
            \node [vertex, label=right:\footnotesize${[}l-d_1+2{]}$, below of=vk+l-3, node distance=1.5cm] (vl-d1+2) {};
            \node [vertex, label=left:\footnotesize${[}l-d_1+1{]}$, below of=vk+l-2, node distance=2.5cm] (vl-d1+1) {};
            \node [vertex, label=right:\footnotesize${[}l-d_1{]}$, below of=vl-d1+2] (vl-d1) {};
            \draw (vk+l) to (vk+l-1);
            \draw (vk+l-2) to (vk+l-1);
            \draw (vk+l-2) to (vk+l-3);
            \draw (vl-d1+1) to (vl-d1+2);
            \draw (vl-d1+1) to (vl-d1);
            \draw[dotted] (vk+l-3)+(202.5:1cm) -- (vk+l-3);
            \draw[dotted] (vl-d1+2)+(157.5:1cm) --(vl-d1+2); 
            \end{tikzpicture}$ &

            $\begin{tikzpicture}[node distance=1cm, auto, baseline=(current  bounding  box.center)]
            \tikzstyle{every label}=[text height=\heightof{1}];
            \tikzstyle{every path}=[->, shorten >= 0.5pt, >=stealth']
            \tikzstyle{vertex}=[circle, fill, inner sep=1.5pt, outer sep=0mm];

            \node [vertex, label=left:\footnotesize${[}k-l+2{]}$] (vk-l+2) at (0,-1) {};
            \node [vertex, label=right:\footnotesize${[}k-l+3{]}$] (vk-l+3) at (1,-0.5) {};
            \node [vertex, label=left:\footnotesize${[}k-l{]}$, below of=vk-l+2] (vk-l) {};
            \node [vertex, label=right:\footnotesize${[}k-l+1{]}$, below of=vk-l+3] (vk-l+1) {};
            \node (dots) at (0.5, 0.25) {\vdots};
            \node [vertex, label=right:\footnotesize${[}m+d_1-l-2{]}$, above of=vk-l+3, node distance=1.5cm] (vm+d1-l-2) {};
            \node [vertex, label=left:\footnotesize${[}m+d_1-l-1{]}$, above of=vk-l+2, node distance=2.5cm] (vm+d1-l-1) {};
            \node [vertex, label=right:\footnotesize${[}m+d_1-l{]}$, above of=vm+d1-l-2] (vm+d1-l) {};
            \draw (vk-l) to (vk-l+1);
            \draw (vk-l+2) to (vk-l+1);
            \draw (vk-l+2) to (vk-l+3);
            \draw (vm+d1-l-1) to (vm+d1-l-2);
            \draw (vm+d1-l-1) to (vm+d1-l);
            \draw[dotted] (vk-l+3)+(157.5:1cm) -- (vk-l+3);
            \draw[dotted] (vm+d1-l-2)+(202.5:1cm) --(vm+d1-l-2); 
            \end{tikzpicture}$ \\

            if $k-1 = d_1 < l \leqslant d_2$ & 
            if $m-1-k = d_1 < l \leqslant d_2$ \\[1cm]

            $\begin{tikzpicture}[node distance=1cm, auto, baseline=(current  bounding  box.center)]
            \tikzstyle{every label}=[text height=\heightof{1}];
            \tikzstyle{every path}=[->, shorten >= 0.5pt, >=stealth']
            \tikzstyle{vertex}=[circle, fill, inner sep=1.5pt, outer sep=0mm];

            \node [vertex, label=left:\footnotesize${[}m-l+d_2-1{]}$] (vm-l+d2-1) at (0,1) {};
            \node [vertex, label=right:\footnotesize${[}m-l+d_2-2{]}$] (vm-l+d2-2) at (1,0.5) {};
            \node [vertex, label=right:\footnotesize${[}m-l+d_2{]}$, above of=vm-l+d2-2] (vm-l+d2) {};
            \node (dots) at (0.5, -0.25) {\vdots};
            \node [vertex, label=right:\footnotesize${[}l-d_1+2{]}$, below of=vm-l+d2-2, node distance=1.5cm] (vl-d1+2) {};
            \node [vertex, label=left:\footnotesize${[}l-d_1+1{]}$, below of=vm-l+d2-1, node distance=2.5cm] (vl-d1+1) {};
            \node [vertex, label=right:\footnotesize${[}l-d_1{]}$, below of=vl-d1+2] (vl-d1) {};
            \draw (vm-l+d2-1) to (vm-l+d2);
            \draw (vm-l+d2-1) to (vm-l+d2-2);
            \draw (vl-d1+1) to (vl-d1+2);
            \draw (vl-d1+1) to (vl-d1);
            \draw[dotted] (vm-l+d2-2)+(202.5:1cm) -- (vm-l+d2-2);
            \draw[dotted] (vl-d1+2)+(157.5:1cm) -- (vl-d1+2); 
            \end{tikzpicture}$ &

            $\begin{tikzpicture}[node distance=1cm, auto, baseline=(current  bounding  box.center)]
            \tikzstyle{every label}=[text height=\heightof{1}];
            \tikzstyle{every path}=[->, shorten >= 0.5pt, >=stealth']
            \tikzstyle{vertex}=[circle, fill, inner sep=1.5pt, outer sep=0mm];

            \node [vertex, label=left:\footnotesize${[}m-l+d_1-1{]}$] (vm-l+d1-1) at (0,1) {};
            \node [vertex, label=right:\footnotesize${[}m-l+d_1-2{]}$] (vm-l+d1-2) at (1,0.5) {};
            \node [vertex, label=right:\footnotesize${[}m-l+d_1{]}$, above of=vm-l+d1-2] (vm-l+d1) {};
            \node (dots) at (0.5, -0.25) {\vdots};
            \node [vertex, label=right:\footnotesize${[}l-d_2+2{]}$, below of=vm-l+d1-2, node distance=1.5cm] (vl-d2+2) {};
            \node [vertex, label=left:\footnotesize${[}l-d_2+1{]}$, below of=vm-l+d1-1, node distance=2.5cm] (vl-d2+1) {};
            \node [vertex, label=right:\footnotesize${[}l-d_2{]}$, below of=vl-d2+2] (vl-d2) {};
            \draw (vm-l+d1-1) to (vm-l+d1);
            \draw (vm-l+d1-1) to (vm-l+d1-2);
            \draw (vl-d2+1) to (vl-d2+2);
            \draw (vl-d2+1) to (vl-d2);
            \draw[dotted] (vm-l+d1-2)+(202.5:1cm) -- (vm-l+d1-2);
            \draw[dotted] (vl-d2+2)+(157.5:1cm) -- (vl-d2+2); 
            \end{tikzpicture}$ \\

            if $k-1 = d_1 \leqslant d_2 < l \leqslant m-1$ & 
            if $m-k-1 = d_1 \leqslant d_2 < l \leqslant m-1$ \\[1cm]
         \end{tabular}
      \end{adjustbox}
      \caption{Graphs describing $F_{\widehat{l}_q} B_w$}
      \label{tabMinComp}
   \end{table}
   
   Before proving the claim we want to make a few remarks. First note that as long as $l$ 
   satisfies $0 \leqslant l < d_1$ (resp. $l = d_1 \neq d_2$) the number of indecomposable Soergel 
   bimodules occuring in the minimal complex increases by $2$ (resp. $1$) when $l$ is increased 
   by $1$. For $l$ in the range $d_1 < l < d_2$ (resp. $l = d_1 = d_2$) the number of indecomposable Soergel 
   bimodules occuring in the minimal complex is constantly equal to $2d_1+2$ (resp. $2d_1+1=m-1$ in 
   this case). As long as $l$ satisfies $d_2 < l < m-1$ (resp. $d_1 \neq d_2 = l$) the number 
   of indecomposable Soergel bimodules occuring in the minimal complex decreases by $2$ 
   (resp. $1$) when $l$ is increased by $1$. Thus for $l=m-1$ there remains only one 
   indecomposable Soergel bimodule in the right column of the graph (i.e.  
	 with grading shift $(1)$ in the first cohomological degree of the corresponding minimal 
	 complex) (as $m-1 - d_2 -1 = d_1$) and it corresponds to $\widetilde{w}$. 
	 Therefore the claim implies the statement of the lemma.

	 The reader should think of these graphs as describing a one-dimensional wave oscillating
	 in a bounded region (i.e. the graph $\Gamma$) at discrete time values $t=0, \dots, m-1$. 
	 At $t=0$ the wave is stimulated in a point (i.e. the vertex $B_w$ of $\Gamma$). 
	 From $t=n$ to $t=n+1$ all wave crests transform into wave troughs and vice versa.
	 The evolution of the wave can be divided into three periods: First for $0 
	 \leqslant t \leqslant d_1$ the wave front travels in both directions and thus the agitated 
	 region grows until the first wave front hits the nearest boundary where it gets reflected.
	 Then for $d_1 < t \leqslant d_2$ the resulting wave packet propagates towards the other 
	 boundary where the wave front again gets reflected in such a way that the superposition
	 of the resulting wave fronts leads to extinction. Finally for $d_2 < t \leqslant m-1$ the
	 agitated region shrinks and the wave eventually vanishes. Each period corresponds to
	 one row in the above table.
	
   Finally it should be noted that for $0 < l < m-1$ all the sources (resp. sinks) 
   in the minimal complex describing $F_{\widehat{l}_q} B_w$ are mapped under $\pi_s$ 
   to the unique element in $\desc{L}(\widehat{l}_q)$ (resp. $\{r, q\} \setminus 
   \desc{L}(\widehat{l}_q)=\desc{L}(\widehat{l+1}_q)$).
   
   The claim is proven by induction on $l$ and explicit calculation. For $l=0$ nothing has
   to be checked and the case $l=1$ follows from \cref{lemFirstStep}.
\end{proof}

\begin{ex}
   We want to illustrate the claim of the last proof in the case $m_{q, r} = 8$.
%
   We get for $\Gamma = \pi_s^{-1}(\{q,r\})$
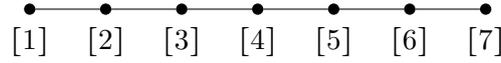
   \[ \begin{tikzpicture}[node distance=1cm, auto, baseline=(current  bounding  box.center)]
      \tikzstyle{every label}=[text height=\heightof{(1)}];
      \tikzstyle{vertex}=[circle, fill, inner sep=1.5pt, outer sep=0mm];
      \tikzstyle{autom}=[<->, shorten >= 2pt, shorten <= 2pt, >=stealth', bend left, red];
      
      \node [vertex, label=below:\small${[}1{]}$] (v1) at (0,0) {};
      \node [vertex, label=below:\small${[}2{]}$, right of=v1] (v2) {};
      \node [vertex, label=below:\small${[}3{]}$, right of=v2] (v3) {};
      \node [vertex, label=below:\small${[}4{]}$, right of=v3] (v4) {};
      \node [vertex, label=below:\small${[}5{]}$, right of=v4] (v5) {};
      \node [vertex, label=below:\small${[}6{]}$, right of=v5] (v6) {};
      \node [vertex, label=below:\small${[}7{]}$, right of=v6] (v7) {};
      \draw (v1) to (v2);
      \draw (v2) to (v3);
      \draw (v3) to (v4);
      \draw (v4) to (v5);
      \draw (v5) to (v6);
      \draw (v6) to (v7);
   \end{tikzpicture} \]
   where $\pi_s^{-1}(r) = \{[1],[3],[5],[7]\}$ and $\pi_s^{-1}(q) = \{[2],[4],[6]\}$.
   For $k=3$ we have the following graphs for $F_{\widehat{l}_q} B_{[k]}$ in the first phase:
   \begin{center}
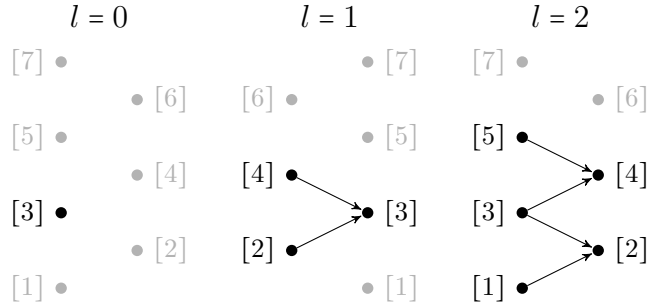
      \begin{tabular}{ccc}
         $l=0$ & $l=1$ & $l=2$ \\
      
         $\begin{tikzpicture}[node distance=1cm, auto, baseline=(current  bounding  box.center)]
            \tikzstyle{every label}=[text height=\heightof{1}];
            \tikzstyle{every path}=[->, shorten >= 0.5pt, >=stealth']
            \tikzstyle{vertex}=[circle, fill, inner sep=1.5pt, outer sep=0mm];
            \tikzstyle{non-ex}=[black!30, fill=black!30, label={[black!30]#1}];
            
            \node [vertex, label=left:\footnotesize${[}3{]}$] (v3) at (0,-1) {};
            \node [vertex, non-ex=right:\footnotesize${[}4{]}$] (v4) at (1,-0.5) {};
            \node [vertex, non-ex=left:\footnotesize${[}1{]}$, below of=v3] (v1) {};
            \node [vertex, non-ex=right:\footnotesize${[}2{]}$, below of=v4] (v2) {};
            \node [vertex, non-ex=left:\footnotesize${[}5{]}$, above of=v3] (v5) {};
            \node [vertex, non-ex=right:\footnotesize${[}6{]}$, above of=v4] (v6) {};
            \node [vertex, non-ex=left:\footnotesize${[}7{]}$, above of=v5] (v7) {};
         \end{tikzpicture}$ &
         
         $\begin{tikzpicture}[node distance=1cm, auto, baseline=(current  bounding  box.center)]
            \tikzstyle{every label}=[text height=\heightof{1}];
            \tikzstyle{every path}=[->, shorten >= 0.5pt, >=stealth']
            \tikzstyle{vertex}=[circle, fill, inner sep=1.5pt, outer sep=0mm];
            \tikzstyle{non-ex}=[black!30, fill=black!30, label={[black!30]#1}];
            
            \node [vertex, label=right:\footnotesize${[}3{]}$] (v3) at (1,-0.5) {};
            \node [vertex, label=left:\footnotesize${[}4{]}$] (v4) at (0,0) {};
            \node [vertex, non-ex=right:\footnotesize${[}1{]}$, below of=v3] (v1) {};
            \node [vertex, label=left:\footnotesize${[}2{]}$, below of=v4] (v2) {};
            \node [vertex, non-ex=right:\footnotesize${[}5{]}$, above of=v3] (v5) {};
            \node [vertex, non-ex=left:\footnotesize${[}6{]}$, above of=v4] (v6) {};
            \node [vertex, non-ex=right:\footnotesize${[}7{]}$, above of=v5] (v7) {};
            \draw (v2) to (v3);
            \draw (v4) to (v3);
         \end{tikzpicture}$ &
         
         $\begin{tikzpicture}[node distance=1cm, auto, baseline=(current  bounding  box.center)]
            \tikzstyle{every label}=[text height=\heightof{1}];
            \tikzstyle{every path}=[->, shorten >= 0.5pt, >=stealth']
            \tikzstyle{vertex}=[circle, fill, inner sep=1.5pt, outer sep=0mm];
            \tikzstyle{non-ex}=[black!30, fill=black!30, label={[black!30]#1}];
            
            \node [vertex, label=left:\footnotesize${[}3{]}$] (v3) at (0,-1) {};
            \node [vertex, label=right:\footnotesize${[}4{]}$] (v4) at (1,-0.5) {};
            \node [vertex, label=left:\footnotesize${[}1{]}$, below of=v3] (v1) {};
            \node [vertex, label=right:\footnotesize${[}2{]}$, below of=v4] (v2) {};
            \node [vertex, label=left:\footnotesize${[}5{]}$, above of=v3] (v5) {};
            \node [vertex, non-ex=right:\footnotesize${[}6{]}$, above of=v4] (v6) {};
            \node [vertex, non-ex=left:\footnotesize${[}7{]}$, above of=v5] (v7) {};
            \draw (v1) to (v2);
            \draw (v3) to (v2);
            \draw (v3) to (v4);
            \draw (v5) to (v4);
         \end{tikzpicture}$
      \end{tabular}
   \end{center}
   The second phase looks like:
   \begin{center}
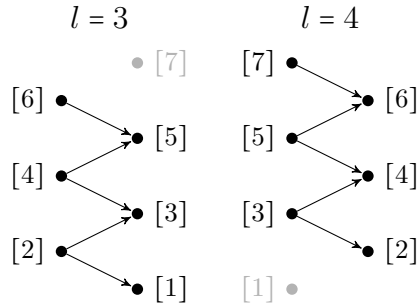
      \begin{tabular}{cccc}
         $l=3$ & $l=4$ \\
         
         $\begin{tikzpicture}[node distance=1cm, auto, baseline=(current  bounding  box.center)]
            \tikzstyle{every label}=[text height=\heightof{1}];
            \tikzstyle{every path}=[->, shorten >= 0.5pt, >=stealth']
            \tikzstyle{vertex}=[circle, fill, inner sep=1.5pt, outer sep=0mm];
            \tikzstyle{non-ex}=[black!30, fill=black!30, label={[black!30]#1}];
            
            \node [vertex, label=right:\footnotesize${[}3{]}$] (v3) at (1,-0.5) {};
            \node [vertex, label=left:\footnotesize${[}4{]}$] (v4) at (0,0) {};
            \node [vertex, label=right:\footnotesize${[}1{]}$, below of=v3] (v1) {};
            \node [vertex, label=left:\footnotesize${[}2{]}$, below of=v4] (v2) {};
            \node [vertex, label=right:\footnotesize${[}5{]}$, above of=v3] (v5) {};
            \node [vertex, label=left:\footnotesize${[}6{]}$, above of=v4] (v6) {};
            \node [vertex, non-ex=right:\footnotesize${[}7{]}$, above of=v5] (v7) {};
            \draw (v2) to (v1);
            \draw (v2) to (v3);
            \draw (v4) to (v3);
            \draw (v4) to (v5);
            \draw (v6) to (v5);
         \end{tikzpicture}$ &
         
         $\begin{tikzpicture}[node distance=1cm, auto, baseline=(current  bounding  box.center)]
            \tikzstyle{every label}=[text height=\heightof{1}];
            \tikzstyle{every path}=[->, shorten >= 0.5pt, >=stealth']
            \tikzstyle{vertex}=[circle, fill, inner sep=1.5pt, outer sep=0mm];
            \tikzstyle{non-ex}=[black!30, fill=black!30, label={[black!30]#1}];
            
            \node [vertex, label=left:\footnotesize${[}3{]}$] (v3) at (0,-1) {};
            \node [vertex, label=right:\footnotesize${[}4{]}$] (v4) at (1,-0.5) {};
            \node [vertex, non-ex=left:\footnotesize${[}1{]}$, below of=v3] (v1) {};
            \node [vertex, label=right:\footnotesize${[}2{]}$, below of=v4] (v2) {};
            \node [vertex, label=left:\footnotesize${[}5{]}$, above of=v3] (v5) {};
            \node [vertex, label=right:\footnotesize${[}6{]}$, above of=v4] (v6) {};
            \node [vertex, label=left:\footnotesize${[}7{]}$, above of=v5] (v7) {};
            \draw (v3) to (v2);
            \draw (v3) to (v4);
            \draw (v5) to (v4);
            \draw (v5) to (v6);
            \draw (v7) to (v6);
         \end{tikzpicture}$
      \end{tabular}
   \end{center}
   And in the last phase we get:
   \begin{center}
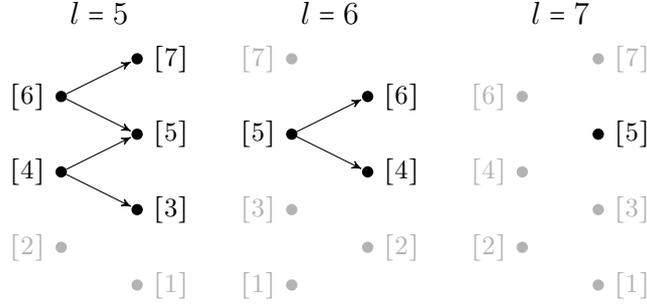
      \begin{tabular}{ccc}
         $l=5$ & $l=6$ & $l=7$ \\
         
         $\begin{tikzpicture}[node distance=1cm, auto, baseline=(current  bounding  box.center)]
            \tikzstyle{every label}=[text height=\heightof{1}];
            \tikzstyle{every path}=[->, shorten >= 0.5pt, >=stealth']
            \tikzstyle{vertex}=[circle, fill, inner sep=1.5pt, outer sep=0mm];
            \tikzstyle{non-ex}=[black!30, fill=black!30, label={[black!30]#1}];
            
            \node [vertex, label=right:\footnotesize${[}3{]}$] (v3) at (1,-0.5) {};
            \node [vertex, label=left:\footnotesize${[}4{]}$] (v4) at (0,0) {};
            \node [vertex, non-ex=right:\footnotesize${[}1{]}$, below of=v3] (v1) {};
            \node [vertex, non-ex=left:\footnotesize${[}2{]}$, below of=v4] (v2) {};
            \node [vertex, label=right:\footnotesize${[}5{]}$, above of=v3] (v5) {};
            \node [vertex, label=left:\footnotesize${[}6{]}$, above of=v4] (v6) {};
            \node [vertex, label=right:\footnotesize${[}7{]}$, above of=v5] (v7) {};
            \draw (v4) to (v3);
            \draw (v4) to (v5);
            \draw (v6) to (v5);
            \draw (v6) to (v7);
         \end{tikzpicture}$ &
         
         $\begin{tikzpicture}[node distance=1cm, auto, baseline=(current  bounding  box.center)]
            \tikzstyle{every label}=[text height=\heightof{1}];
            \tikzstyle{every path}=[->, shorten >= 0.5pt, >=stealth']
            \tikzstyle{vertex}=[circle, fill, inner sep=1.5pt, outer sep=0mm];
            \tikzstyle{non-ex}=[black!30, fill=black!30, label={[black!30]#1}];
            
            \node [vertex, non-ex=left:\footnotesize${[}3{]}$] (v3) at (0,-1) {};
            \node [vertex, label=right:\footnotesize${[}4{]}$] (v4) at (1,-0.5) {};
            \node [vertex, non-ex=left:\footnotesize${[}1{]}$, below of=v3] (v1) {};
            \node [vertex, non-ex=right:\footnotesize${[}2{]}$, below of=v4] (v2) {};
            \node [vertex, label=left:\footnotesize${[}5{]}$, above of=v3] (v5) {};
            \node [vertex, label=right:\footnotesize${[}6{]}$, above of=v4] (v6) {};
            \node [vertex, non-ex=left:\footnotesize${[}7{]}$, above of=v5] (v7) {};
            \draw (v5) to (v4);
            \draw (v5) to (v6);
         \end{tikzpicture}$ &
         
         $\begin{tikzpicture}[node distance=1cm, auto, baseline=(current  bounding  box.center)]
            \tikzstyle{every label}=[text height=\heightof{1}];
            \tikzstyle{every path}=[->, shorten >= 0.5pt, >=stealth']
            \tikzstyle{vertex}=[circle, fill, inner sep=1.5pt, outer sep=0mm];
            \tikzstyle{non-ex}=[black!30, fill=black!30, label={[black!30]#1}];
            
            \node [vertex, non-ex=right:\footnotesize${[}3{]}$] (v3) at (1,-0.5) {};
            \node [vertex, non-ex=left:\footnotesize${[}4{]}$] (v4) at (0,0) {};
            \node [vertex, non-ex=right:\footnotesize${[}1{]}$, below of=v3] (v1) {};
            \node [vertex, non-ex=left:\footnotesize${[}2{]}$, below of=v4] (v2) {};
            \node [vertex, label=right:\footnotesize${[}5{]}$, above of=v3] (v5) {};
            \node [vertex, non-ex=left:\footnotesize${[}6{]}$, above of=v4] (v6) {};
            \node [vertex, non-ex=right:\footnotesize${[}7{]}$, above of=v5] (v7) {};
         \end{tikzpicture}$
      \end{tabular}
   \end{center}
   
   For $k=4$ the following graphs describe $F_{\widehat{l}_r} B_{[k]}$ in the first phase:
   \begin{center}
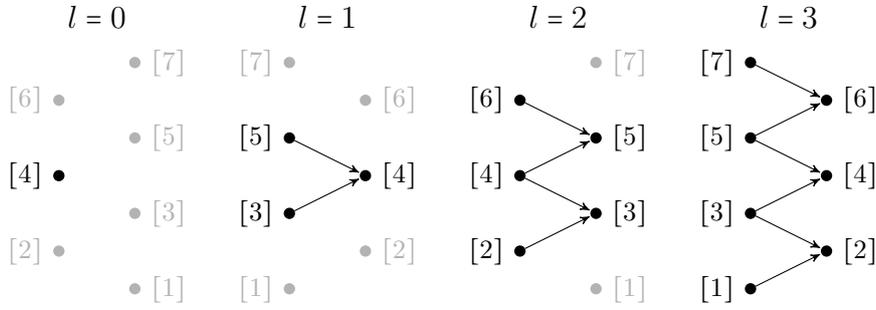
      \begin{tabular}{cccc}
         $l=0$ & $l=1$ & $l=2$ & $l=3$ \\

         $\begin{tikzpicture}[node distance=1cm, auto, baseline=(current  bounding  box.center)]
            \tikzstyle{every label}=[text height=\heightof{1}];
            \tikzstyle{every path}=[->, shorten >= 0.5pt, >=stealth']
            \tikzstyle{vertex}=[circle, fill, inner sep=1.5pt, outer sep=0mm];
            \tikzstyle{non-ex}=[black!30, fill=black!30, label={[black!30]#1}];
            
            \node [vertex, non-ex=right:\footnotesize${[}3{]}$] (v3) at (1,-0.5) {};
            \node [vertex, label=left:\footnotesize${[}4{]}$] (v4) at (0,0) {};
            \node [vertex, non-ex=right:\footnotesize${[}1{]}$, below of=v3] (v1) {};
            \node [vertex, non-ex=left:\footnotesize${[}2{]}$, below of=v4] (v2) {};
            \node [vertex, non-ex=right:\footnotesize${[}5{]}$, above of=v3] (v5) {};
            \node [vertex, non-ex=left:\footnotesize${[}6{]}$, above of=v4] (v6) {};
            \node [vertex, non-ex=right:\footnotesize${[}7{]}$, above of=v5] (v7) {};
         \end{tikzpicture}$ &
         
         $\begin{tikzpicture}[node distance=1cm, auto, baseline=(current  bounding  box.center)]
            \tikzstyle{every label}=[text height=\heightof{1}];
            \tikzstyle{every path}=[->, shorten >= 0.5pt, >=stealth']
            \tikzstyle{vertex}=[circle, fill, inner sep=1.5pt, outer sep=0mm];
            \tikzstyle{non-ex}=[black!30, fill=black!30, label={[black!30]#1}];
            
            \node [vertex, label=left:\footnotesize${[}3{]}$] (v3) at (0,-1) {};
            \node [vertex, label=right:\footnotesize${[}4{]}$] (v4) at (1,-0.5) {};
            \node [vertex, non-ex=left:\footnotesize${[}1{]}$, below of=v3] (v1) {};
            \node [vertex, non-ex=right:\footnotesize${[}2{]}$, below of=v4] (v2) {};
            \node [vertex, label=left:\footnotesize${[}5{]}$, above of=v3] (v5) {};
            \node [vertex, non-ex=right:\footnotesize${[}6{]}$, above of=v4] (v6) {};
            \node [vertex, non-ex=left:\footnotesize${[}7{]}$, above of=v5] (v7) {};
            \draw (v3) to (v4);
            \draw (v5) to (v4);
         \end{tikzpicture}$ &
         
         $\begin{tikzpicture}[node distance=1cm, auto, baseline=(current  bounding  box.center)]
            \tikzstyle{every label}=[text height=\heightof{1}];
            \tikzstyle{every path}=[->, shorten >= 0.5pt, >=stealth']
            \tikzstyle{vertex}=[circle, fill, inner sep=1.5pt, outer sep=0mm];
            \tikzstyle{non-ex}=[black!30, fill=black!30, label={[black!30]#1}];
            
            \node [vertex, label=right:\footnotesize${[}3{]}$] (v3) at (1,-0.5) {};
            \node [vertex, label=left:\footnotesize${[}4{]}$] (v4) at (0,0) {};
            \node [vertex, non-ex=right:\footnotesize${[}1{]}$, below of=v3] (v1) {};
            \node [vertex, label=left:\footnotesize${[}2{]}$, below of=v4] (v2) {};
            \node [vertex, label=right:\footnotesize${[}5{]}$, above of=v3] (v5) {};
            \node [vertex, label=left:\footnotesize${[}6{]}$, above of=v4] (v6) {};
            \node [vertex, non-ex=right:\footnotesize${[}7{]}$, above of=v5] (v7) {};
            \draw (v2) to (v3);
            \draw (v4) to (v3);
            \draw (v4) to (v5);
            \draw (v6) to (v5);
         \end{tikzpicture}$ &

         $\begin{tikzpicture}[node distance=1cm, auto, baseline=(current  bounding  box.center)]
            \tikzstyle{every label}=[text height=\heightof{1}];
            \tikzstyle{every path}=[->, shorten >= 0.5pt, >=stealth']
            \tikzstyle{vertex}=[circle, fill, inner sep=1.5pt, outer sep=0mm];
            \tikzstyle{non-ex}=[black!30, fill=black!30, label={[black!30]#1}];
            
            \node [vertex, label=left:\footnotesize${[}3{]}$] (v3) at (0,-1) {};
            \node [vertex, label=right:\footnotesize${[}4{]}$] (v4) at (1,-0.5) {};
            \node [vertex, label=left:\footnotesize${[}1{]}$, below of=v3] (v1) {};
            \node [vertex, label=right:\footnotesize${[}2{]}$, below of=v4] (v2) {};
            \node [vertex, label=left:\footnotesize${[}5{]}$, above of=v3] (v5) {};
            \node [vertex, label=right:\footnotesize${[}6{]}$, above of=v4] (v6) {};
            \node [vertex, label=left:\footnotesize${[}7{]}$, above of=v5] (v7) {};
            \draw (v1) to (v2);
            \draw (v3) to (v2);
            \draw (v3) to (v4);
            \draw (v5) to (v4);
            \draw (v5) to (v6);
            \draw (v7) to (v6);
         \end{tikzpicture}$
      \end{tabular}
   \end{center}
   
   The second phase completely collapses and in the third phase we have:
   \begin{center}
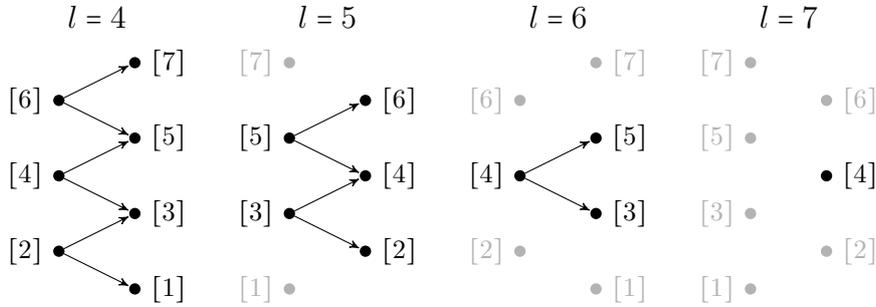
      \begin{tabular}{cccc}
         $l=4$ & $l=5$ & $l=6$ & $l=7$ \\

         $\begin{tikzpicture}[node distance=1cm, auto, baseline=(current  bounding  box.center)]
            \tikzstyle{every label}=[text height=\heightof{1}];
            \tikzstyle{every path}=[->, shorten >= 0.5pt, >=stealth']
            \tikzstyle{vertex}=[circle, fill, inner sep=1.5pt, outer sep=0mm];
            \tikzstyle{non-ex}=[black!30, fill=black!30, label={[black!30]#1}];
            
            \node [vertex, label=right:\footnotesize${[}3{]}$] (v3) at (1,-0.5) {};
            \node [vertex, label=left:\footnotesize${[}4{]}$] (v4) at (0,0) {};
            \node [vertex, label=right:\footnotesize${[}1{]}$, below of=v3] (v1) {};
            \node [vertex, label=left:\footnotesize${[}2{]}$, below of=v4] (v2) {};
            \node [vertex, label=right:\footnotesize${[}5{]}$, above of=v3] (v5) {};
            \node [vertex, label=left:\footnotesize${[}6{]}$, above of=v4] (v6) {};
            \node [vertex, label=right:\footnotesize${[}7{]}$, above of=v5] (v7) {};
            \draw (v2) to (v1);
            \draw (v2) to (v3);
            \draw (v4) to (v3);
            \draw (v4) to (v5);
            \draw (v6) to (v5);
            \draw (v6) to (v7);
         \end{tikzpicture}$ &
         
         $\begin{tikzpicture}[node distance=1cm, auto, baseline=(current  bounding  box.center)]
            \tikzstyle{every label}=[text height=\heightof{1}];
            \tikzstyle{every path}=[->, shorten >= 0.5pt, >=stealth']
            \tikzstyle{vertex}=[circle, fill, inner sep=1.5pt, outer sep=0mm];
            \tikzstyle{non-ex}=[black!30, fill=black!30, label={[black!30]#1}];
            
            \node [vertex, label=left:\footnotesize${[}3{]}$] (v3) at (0,-1) {};
            \node [vertex, label=right:\footnotesize${[}4{]}$] (v4) at (1,-0.5) {};
            \node [vertex, non-ex=left:\footnotesize${[}1{]}$, below of=v3] (v1) {};
            \node [vertex, label=right:\footnotesize${[}2{]}$, below of=v4] (v2) {};
            \node [vertex, label=left:\footnotesize${[}5{]}$, above of=v3] (v5) {};
            \node [vertex, label=right:\footnotesize${[}6{]}$, above of=v4] (v6) {};
            \node [vertex, non-ex=left:\footnotesize${[}7{]}$, above of=v5] (v7) {};
            \draw (v3) to (v4);
            \draw (v3) to (v2);
            \draw (v5) to (v4);
            \draw (v5) to (v6);
         \end{tikzpicture}$ &
         
         $\begin{tikzpicture}[node distance=1cm, auto, baseline=(current  bounding  box.center)]
            \tikzstyle{every label}=[text height=\heightof{1}];
            \tikzstyle{every path}=[->, shorten >= 0.5pt, >=stealth']
            \tikzstyle{vertex}=[circle, fill, inner sep=1.5pt, outer sep=0mm];
            \tikzstyle{non-ex}=[black!30, fill=black!30, label={[black!30]#1}];
            
            \node [vertex, label=right:\footnotesize${[}3{]}$] (v3) at (1,-0.5) {};
            \node [vertex, label=left:\footnotesize${[}4{]}$] (v4) at (0,0) {};
            \node [vertex, non-ex=right:\footnotesize${[}1{]}$, below of=v3] (v1) {};
            \node [vertex, non-ex=left:\footnotesize${[}2{]}$, below of=v4] (v2) {};
            \node [vertex, label=right:\footnotesize${[}5{]}$, above of=v3] (v5) {};
            \node [vertex, non-ex=left:\footnotesize${[}6{]}$, above of=v4] (v6) {};
            \node [vertex, non-ex=right:\footnotesize${[}7{]}$, above of=v5] (v7) {};
            \draw (v4) to (v3);
            \draw (v4) to (v5);
         \end{tikzpicture}$ &

         $\begin{tikzpicture}[node distance=1cm, auto, baseline=(current  bounding  box.center)]
            \tikzstyle{every label}=[text height=\heightof{1}];
            \tikzstyle{every path}=[->, shorten >= 0.5pt, >=stealth']
            \tikzstyle{vertex}=[circle, fill, inner sep=1.5pt, outer sep=0mm];
            \tikzstyle{non-ex}=[black!30, fill=black!30, label={[black!30]#1}];
            
            \node [vertex, non-ex=left:\footnotesize${[}3{]}$] (v3) at (0,-1) {};
            \node [vertex, label=right:\footnotesize${[}4{]}$] (v4) at (1,-0.5) {};
            \node [vertex, non-ex=left:\footnotesize${[}1{]}$, below of=v3] (v1) {};
            \node [vertex, non-ex=right:\footnotesize${[}2{]}$, below of=v4] (v2) {};
            \node [vertex, non-ex=left:\footnotesize${[}5{]}$, above of=v3] (v5) {};
            \node [vertex, non-ex=right:\footnotesize${[}6{]}$, above of=v4] (v6) {};
            \node [vertex, non-ex=left:\footnotesize${[}7{]}$, above of=v5] (v7) {};
         \end{tikzpicture}$
      \end{tabular}
   \end{center}
\end{ex}

The main goal of this section is to prove the following theorem:

\begin{thm}
   \label{thmGarsideAct}
   Let $\sigma = w_m w_{m-1} \dots w_1 \in \brm[(W, S)]$ be a non-trivial braid in Garside normal form.
   Set $B = \bigoplus_{w \in \lcellFixedRightDesc{s}} B_w$. Then the following holds:
   \begin{enumerate}
      \item If $k$ is maximal such that $\pH{k}(F_{\sigma}(B))$ is non-zero, then $k = m$.
      \item Let $\lcellFixedRightDesc{s}^{\sigma}$ be the set of Coxeter elements in $\lcellFixedRightDesc{s}$ indexing
            an anchor in $\pH{m}(F_{\sigma}(B))$. Then $\pi_s$ induces a surjection from 
            $\lcellFixedRightDesc{s}^{\sigma}$ onto $\desc{L}(w_m)$.
   \end{enumerate}
   If in addition $(W, S)$ is a Coxeter group of simply-laced type whose Coxeter graph is a tree, 
   $\pi_s$ gives a bijection between $\lcellFixedRightDesc{s}^{\sigma}$ and $\desc{L}(w_m)$.
\end{thm}
\begin{proof}
   To simplify notation, we will identify $W$ with the set of reduced braids $\varphi(w) \subset
   \brm[(W, S)]$. We will state explicitly if a certain identity only holds in the Coxeter
   group.
  
   We will prove by simultaneous induction on $m$ and on $l(w_m)$ the following three statements from
   which the theorem follows:
   \begin{enumerate}
      \item[P)] $\pH{k}(F_{\sigma}(B)) = 0$ for $k > m$.
      \item[L)] For all $t \in \desc{L}(w_m)$ there exists an element $w \in \pi_s^{-1}(t)$ such that
                the indecomposable Soergel bimodule $B_{w}$ is an anchor of $F_{\sigma}(B)$.                
      \item[A)] Any anchor $B_w$ of $F_{\sigma}(B)$ satisfies $\pi_s(w) \in \desc{L}(w_m)$. 
   \end{enumerate}
   
   Note that under the additional assumptions that $(W, S)$ is a Coxeter group of simply-laced 
   type whose Coxeter graph is a tree, $\pi_s$ is injective and thus the last statement
   of the theorem follows from L).
   
   The proof strategy actually is as follows: We apply induction on $m$; in the base case
   as well as in the inductive step we apply induction on $l(w_m)$. Since the inductive step
   for the induction on $l(w_m)$ does not depend on whether we are in the base case
   or the inductive step for the induction on $m$, we treat these cases together. Introduce 
   the following notation for $x \in \{P, L, A\}$:
   \begin{description}
      \item[$X(n,l)$:] Statement x) holds for all $\beta \in \brm$ with Garside length $\leqslant n$ and
                       length of the final Garside factor $\leqslant l$.
      \item[$X(n,\infty)$: ] $X(n, l)$ for all $l \in \N$. 
   \end{description}
   
   In the case $m=1$ and $w_m = t \in S$ all statements above follow immediately from \cref{lemFirstStep}.
   In other words: \cref{lemFirstStep} $\Rightarrow P(1,1), L(1,1), A(1,1)$

   Suppose by induction that the statements above hold for any positive braid with fewer Garside factors than $\sigma$
   or with $m$ Garside factors and shorter final (i.e. $w_m$) Garside factor than $\sigma$.
   
   For simplicity of notation set $\beta = w_{m-1} w_{m-2} \dots w_1$. Since being a Garside normal form
   can be checked locally (see \cref{propGarsideNormFormLocal}), $\beta$ is in Garside normal form.
   
   First, we prove P):
   
   \begin{lem}
      $m > 1 + P(m-1, \infty) \Longrightarrow P(m, 1)$
   \end{lem}
   \begin{proof} Write $w_m = r \in S$.
   By induction on the number of Garside factors we know that $\pH{k}(F_{\beta}(B)) = 0$ for $k > m-1$. 
   Applying $F_r$, the left handed multiplication formula from \cref{eqnLHMultForm} implies that
   the grading of any indecomposable Soergel bimodule gets shifted up at most by $1$ (compare with
   \cref{lemFsLeqComp}). It follows that $\pH{k}(F_{\sigma}(B)) = 0$ for $k > m$.
   \end{proof}
   
   \begin{lem}
      $l(w_m) > 1 + P(m, l(w_m)-1) + A(m, l(w_m)-1) \Longrightarrow P(m, l(w_m))$
   \end{lem}
   \begin{proof}
   Write $w_m = t u$ with $t \in S$ and $l(u) = l(w_m) - 1 \geqslant 1$.
   By induction on the length of the final Garside factor, we know that the maximal $k$ such that 
   $\pH{k}(F_{u\beta}(B)) \neq 0$ satisfies $k \leqslant m$. Since $t u > u$ in the Coxeter group, 
   it follows that $t \nin \desc{L}(u)$. Again by induction on the length of the final 
   Garside factor, we can apply A) to conclude that no element in $\pi_s^{-1}(t)$ indexes 
   an anchor of $\pH{k}(F_{u\beta}(B))$. \Cref{corNoAnchorMaxPH} shows that $\pH{k+1}(F_{\sigma}(B)) 
   \cong \pH{1}(F_t \pH{k}(F_{u\beta}(B)))$ vanishes.
   \end{proof}

   Next, we will consider L) and A) and prove both $L(m, 1)$ and $A(m, 1)$ at the same time.
   
   \begin{lem}
      $m > 1 + P(m-1, \infty) + L(m-1, \infty) \Longrightarrow L(m, 1) + A(m, 1)$
   \end{lem}
   \begin{proof} Write $w_m = t \in S$.
   \Cref{lemGarsideNormFormChar} shows that $\desc{R}(w_m) \subseteq \desc{L}(w_{m-1})$. Thus
   $t$ lies in $\desc{L}(w_{m-1})$. By induction on the number of Garside factors, we may apply 
   L) to see that there exists $w \in \pi_s^{-1}(t)$ such that $B_{w}$ is an anchor of $F_{\beta}(B)$. 
   Let $k \in \Z$ be maximal such that $\pH{k}(F_{\beta}(B))$ is non-zero. In addition, induction on the number
   of Garside factors applied to P) implies that $k \leqslant m-1$. \Cref{corAnchorMaxPH} shows that
   $\pH{k+1}(F_{\sigma}(B))$ is non-zero and that the indecomposable Soergel bimodules 
   occurring in $\pH{k+1}(F_{\sigma}(B))$ are indexed by elements in $\pi_s^{-1}(t)$. 
   Since $t$ is the only element in $\desc{L}(w_m)$, L) and A) hold in this case.
   \end{proof} 
	
	Then we show the inductive step for $A$ in the induction on $l(w_m)$.
   \begin{lem}
      $l(w_m) > 1 + A(m, l(w_m)-1) \; \Longrightarrow A(m, l(w_m))$.
   \end{lem}
   \begin{proof}
   We show A) by showing its contrapositive. Fix $r \in S \setminus \desc{L}(w_m)$.
   Let $w \in \pi_s^{-1}(r)$ be arbitrary. We want to show that $B_{w}$ is not an anchor 
   of $F_{\sigma}(B)$. Write $w_m = t z$ with $t \in S$ and $l(z) = l(w_m) - 1 \geqslant 1$ 
   in the Coxeter group. It follows that $t \neq r$. As before for $1 \leqslant k \leqslant m_{r, t}$
   denote by  $\prescript{}{t}{\widehat{k}} = t r t \dots$ (resp. $ \widehat{k}_t = \dots trt$) the 
   alternating word in $r$ and $t$ of length $k$ having $t$ in its left (resp. right) 
   descent set. Similarly for $\prescript{}{r}{\widehat{k}}$ and $\widehat{k}_r$.

   \emph{Case 1: $m_{r,t} = 2$ (i.e. $rt = tr$).}
   This implies that $r$ does not lie in the left descent set of $z$. 
   Let $k$ be maximal such that $\pH{k}(F_{z\beta}(B))$ is not zero.
   By induction on the length of the final Garside factor, it follows
   that for all $v \in \pi_s^{-1}(\{r, t\})$ the indecomposable Soergel bimodule
   $B_{v}$ is not an anchor of $\pH{k}(F_{z\beta}(B))$. Therefore 
   \[\Hom_{K^b(\catFixedRightDesc{s})}(F_{z\beta}(B), B_{w}(n-k)[-n])\]
   vanishes for all $n \in \Z$. As there are no edges between $\pi_s^{-1}(t)$ and 
   $w$ in $\Gamma_s$ \cref{lemFirstStep} implies that $F_t B_{w}$ is isomorphic
   to $B_{w}(1)[-1]$. Applying the autoequivalence $F_t$ to both terms of the 
   $\Hom$-space above implies that 
   \[\Hom_{K^b(\catFixedRightDesc{s})}(F_{\sigma}(B), B_{w}(n-k)[-n])\]
   still vanishes for all $n \in \Z$. Using \cref{corNoAnchorMaxPH} for $t$ shows that
   $\pH{k}(F_{\sigma}(B))$ is still the highest non-zero perverse cohomology group.   
   Therefore we see that $B_{w}$ is not an anchor of $F_{\sigma}(B)$.
   
   In the following three cases we have $l < m_{r, t}$ due to $r \nin \desc{L}(w_m)$.
   
   \emph{Case 2: $m_{r, t} \geqslant 3$, $w_m = \prescript{}{t}{\widehat{l}}u$ with $u$ minimal
   in its left $\langle r, t \rangle$-coset, $1\leqslant l < m_{r, t} - 1$ and $l(u) \geqslant 1$.}
   Then $r$ and $t$ are not in $\desc{L}(u)$ and thus by induction on the length of the final Garside 
   factor, for all $v \in \pi_s^{-1}(\{r, t\})$ the indecomposable Soergel bimodule $B_v$ 
   is not an anchor in $F_{u\beta}(B)$. Let $k$ be maximal such  that $\pH{k}(F_{u\beta}(B))$ is non-zero. 
   Therefore \[\Hom_{K^b(\catFixedRightDesc{s})}(F_{u\beta}(B), \bigoplus_{v \in \pi_s^{-1}(\{r, t\})} 
   B_v(n-k)[-n])\] vanishes for all $n \in \Z$. Applying the autoequivalence $F_{\prescript{}{t}{\widehat{l}}}$
   shows that 
   \[\Hom_{K^b(\catFixedRightDesc{s})}(F_{\sigma}(B), \bigoplus_{v \in \pi_s^{-1}(\{r, t\})} 
   F_{\prescript{}{t}{\widehat{l}}} B_v(n-k)[-n])\] 
   still vanishes for all $n \in \Z$.
   From the claim in the proof of \cref{lemBraidStep} we know what each summand in 
   \[\bigoplus_{v \in \pi_s^{-1}(\{r, t\})} F_{\prescript{}{t}{\widehat{l}}} B_v(n-k)[-n]\] 
   looks like and that we may choose $\widetilde{w} \in \pi_s^{-1}(\{r, t\}) \subseteq 
   \lcellFixedRightDesc{s}$ such that $B_w (1)[-1]$ occurs as summand in the first 
   cohomological degree of $F_{\prescript{}{t}{\widehat{l}}} B_{\widetilde{w}}$. 
   In particular there exist sets $I_t \subseteq \pi_s^{-1}(t)$ and $w \nin I_r 
   \subseteq \pi_s^{-1}(r)$ such that $F_{\prescript{}{t}{\widehat{l}}} B_{\widetilde{w}}$ is
   isomorphic to a perverse two-term minimal complex of the following form:
   \[\dots \longrightarrow 0 \longrightarrow \bigoplus_{v \in I_t} B_v \longrightarrow B_w(1) \oplus 
   \bigoplus_{v \in I_r}B_v(1) \longrightarrow 0 \longrightarrow \dots\]
   
  It follows that $B_w$ cannot be an anchor of $F_{\sigma}(B)$ as any non-zero map from 
  $F_{\sigma}(B)$ to $B_w(n-k)[-n]$ for some $n\in \Z$ would induce a non-zero map to 
  $F_{\prescript{}{t}{\widehat{l}}} B_{\widetilde{w}}$ which is a direct summand of 
  \[\bigoplus_{v \in \pi_s^{-1}(\{r, t\})} F_{\prescript{}{t}{\widehat{l}}} B_v(n-1-k)[-n+1]\]
  by post-composition with the class of the following non-zero map in $K^b(\catFixedRightDesc{s})$:
  \resizebox{\textwidth}{!}{$ \begin{xy}
     \xymatrix{
         {\dots} \ar[r] & 0 \ar[r] \ar[d] & 0 \ar[r] \ar[d] 
            & B_w(n-k) \ar[r] \ar[d]^-{ \left( 
               \begin{smallmatrix}
                  \id \\ 0 \\ \vdots \\ 0
               \end{smallmatrix} \right) } 
            & 0 \ar[r] \ar[d] & {\dots} \\
         {\dots} \ar[r] & 0 \ar[r] & {\bigoplus_{v \in I_t} B_v(n-1-k)} \ar[r] 
            & B_w(n-k) \oplus \bigoplus_{v \in I_r }B_v(n-k) \ar[r] & 0 \ar[r] & \dots }
     \end{xy} $ }
   where $B_w(n-k)$ sits in cohomological degree $n$. Indeed, first use repeatedly \cref{corNoAnchorMaxPH} 
   to see that $\pH{k}(F_{\sigma}B)$ still is the highest non-zero perverse cohomology 
   group of $F_{\sigma}B$. Then note that the composition of any representative of the non-zero map 
   $F_{\sigma}(B) \rightarrow B_w(n-k)[-n]$ in $K^b(\catFixedRightDesc{s})$ with the morphism
   of cochain complexes pictured above is non-zero in $C^b(\catFixedRightDesc{s})$. Finally, use the following
   chain of isomorphisms
   \begin{align*}
      & \Hom_{C^b(\catFixedRightDesc{s})}(\pH{k}(F_{\sigma}B), \bigoplus_{v \in \pi_s^{-1}(\{r, t\})} 
        F_{\prescript{}{t}{\widehat{l}}} B_v (n-1)[-n+1]) \\
      &\cong \Hom_{K^b(\catFixedRightDesc{s})}(\pH{k}(F_{\sigma}B), \bigoplus_{v \in \pi_s^{-1}(\{r, t\})} 
        F_{\prescript{}{t}{\widehat{l}}} B_v (n-1)[-n+1]) \\
      &\cong \Hom_{K^b(\catFixedRightDesc{s})}(F_{\sigma}(B), \bigoplus_{v \in \pi_s^{-1}(\{r, t\})} 
        F_{\prescript{}{t}{\widehat{l}}} B_v (n-1)[-n + 1 -k])
   \end{align*}
   where we used the isomorphism from the proof of \cref{lemAnchorEquiv} in the second step and in 
   the first step we assumed that $\pH{k}(F_{\sigma}B)$ is a minimal complex and thus by Soergel's conjecture
   all homotopies between minimal perverse complexes vanish. We conclude 
   that $B_w$ is not an anchor of $F_{\sigma}(B)$.
      
   \emph{Case 3: $m_{r, t} \geqslant 3$, $w_m = \prescript{}{t}{\widehat{l}}$ with 
   $1 < l \leqslant m_{r,t} - 1$.}
   For some $q \in \{r, t\}$ we have $\prescript{}{t}{\widehat{l}} = \widehat{l}_q$. 
   Let $q' \in \{r, t\} \setminus \{q\}$. Let $k$ be maximal such that 
   $\pH{k}(F_{q\beta}(B))$ is non-zero. From the case $w_m = q$ and $m \geqslant 1$ 
   we know that all indecomposable Soergel bimodules occuring in 
   $\pH{k}(F_{q\beta}(B))$ are indexed by elements in $\pi_s^{-1}(q)$.
   
   By \cref{lemBraidStep} there exists $\widetilde{w} \in \pi_s^{-1}(\{r, t\})$ such that 
   $F_{\prescript{}{t}{(\widehat{m_{r, t} - 1})}} B_{\widetilde{w}}$ is isomorphic to
   $B_w(1)[-1]$. Note that $F_{\prescript{}{q}{\widehat{(m_{r, t}-l)}}} B_{\widetilde{w}}$ has
   sources (resp. sinks) indexed by elements in $\pi_s^{-1}(q)$ (resp. $\pi_s^{-1}(q')$) and since 
   each source has non-zero outgoing components of the differential we get (for degree reasons):
   \[ \Hom_{C^b(\catFixedRightDesc{s})}(\pH{k}(F_{q\beta}(B)), F_{\prescript{}{q}{\widehat{(m_{r, t}-l)}}} 
   B_{\widetilde{w}}(n-1-k)[-n+1+k]) = 0 \qquad \forall n \in \Z \]
   
   Now use the following chain of isomorphisms
   \begin{align*}
      & \Hom_{C^b(\catFixedRightDesc{s})}(\pH{k}(F_{q\beta}(B)), F_{\prescript{}{q}{\widehat{(m_{r, t}-l)}}} 
   B_{\widetilde{w}}(n-1-k)[-n+1+k]) \\
      &\cong \Hom_{K^b(\catFixedRightDesc{s})}(\pH{k}(F_{q\beta}(B)), F_{\prescript{}{q}{\widehat{(m_{r, t}-l)}}} 
   B_{\widetilde{w}}(n-1-k)[-n+1+k])\\
      &\cong \Hom_{K^b(\catFixedRightDesc{s})}(F_{q\beta}(B), F_{\prescript{}{q}{\widehat{(m_{r, t}-l)}}} 
   B_{\widetilde{w}}(n-1-k)[-n+1])\\
      &\cong \Hom_{K^b(\catFixedRightDesc{s})}(F_{\sigma}(B), F_{\prescript{}{t}{\widehat{(m_{r, t}-1)}}} 
   B_{\widetilde{w}}(n-1-k)[-n+1])\\
      &\cong \Hom_{K^b(\catFixedRightDesc{s})}(F_{\sigma}(B), B_w(n-k)[-n])\\
   \end{align*}   
   where for the first step one should note that all homotopies between minimal perverse complexes are $0$, in the second step
   the isomorphism from the proof of \cref{lemAnchorEquiv} is used, in the third step the autoequivalence
   $F_{\prescript{}{t}{\widehat{l-1}}}$ is applied and finally the choice of $\widetilde{w}$ comes into play.
   
   This shows that $B_w$ cannot be an anchor of $F_{\sigma}(B)$ as induction applied to A) and \cref{corNoAnchorMaxPH}
   yield that $\pH{k}(F_{\sigma}(B))$ is still the highest non-zero perverse cohomology group.
      
   \emph{Case 4: $m_{r, t} \geqslant 3$, $w_m = \prescript{}{t}{(\widehat{m_{r, t} - 1})}u$ with $u$ minimal
   in its left $\langle r, t \rangle$-coset and $l(u) \geqslant 1$.}
   Then $r$ and $t$ are not in $\desc{L}(u)$ and thus by induction on the length of the final Garside 
   factor, for all $v \in \pi_s^{-1}(\{r, t\})$ the indecomposable Soergel bimodule $B_v$ 
   is not an anchor in $F_{u\beta}(B)$. By \cref{lemBraidStep} there exists 
   $\widetilde{w} \in \pi_s^{-1}(\{r, t\})$ such that 
   $F_{\prescript{}{t}{(\widehat{m_{r, t} - 1})}} B_{\widetilde{w}}$ is isomorphic to
   $B_w(1)[-1]$. Since \[\Hom_{K^b(\catFixedRightDesc{s})}(F_{u\beta}(B), B_{\widetilde{w}}(n-k)[-n])\]
   vanishes for all $n \in \Z$, we may apply the autoequivalence 
   $F_{\prescript{}{t}{(\widehat{m_{r, t} - 1})}}$ to deduce that \[\Hom_{K^b(\catFixedRightDesc{s})}
   (F_{\sigma}(B), B_w(n-k)[-n])\] is zero for all $n \in \Z$. Induction on the length of the Garside factor 
   together with \cref{corNoAnchorMaxPH} imply that $\pH{k}(F_{\sigma}(B))$ is still the highest
   non-zero perverse cohomology group of $F_{\sigma}(B)$. Therefore $B_w$ is not 
   an anchor of $F_{\sigma}(B)$.
   \end{proof}

   Finally, we prove L):
   \begin{lem}
      $l(w_m) > 1 + L(m, l(w_m)-1) + Am, l(w_m)) \Longrightarrow L(m, l(w_m))$
   \end{lem}
   \begin{proof}
   Fix $r \in \desc{L}(w_m)$ and write $w_m = rtz$ with $t \in S$, $l(z) = l(w_m) - 2$ and $z$ possibly trivial.
   Observe that $r \neq t$. We want to show that $B_w$ is an anchor in $F_{\sigma}(B)$
   for some $w \in \pi_s^{-1}(r)$.
   
   \emph{Case 1: $m_{r,t} = 2$ (i.e. $rt = tr$).}
   By induction on the length of the final Garside factor, $B_w$ is an anchor of $F_{rz\beta}(B)$
   for some $w \in \pi_s^{-1}(r)$. Let $k$ be maximal such that $\pH{k}(F_{rz\beta}(B))$ is non-zero.   
   By definition of an anchor, there exists some $n \in \Z$ such that 
   \[\Hom_{K^b(\catFixedRightDesc{s})}(F_{rz\beta}(B), B_w(n-k)[-n])\] does
   not vanish. Applying the autoequivalence $F_t$ and using that $F_t B_w$
   is isomorphic to $B_w(1)[-1]$ by \cref{lemFirstStep} as there are no edges
   between $\pi_s^{-1}(t)$ and $w$ in $\Gamma_s$, it follows that
   \[\Hom_{K^b(\catFixedRightDesc{s})}(F_{\sigma}(B), B_w(n+1-k)[-n-1])\]
   is non-zero. By induction on the length of the Garside factor applied to A, we know that $B_t$
   is not an anchor of $F_{rz\beta}(B)$. Therefore, \cref{corNoAnchorMaxPH} shows that 
   a new non-zero highest perverse cohomology group is not created when applying $F_t$ and 
   thus $B_w$ is an anchor of $F_{\sigma}(B)$.
   
   \emph{Case 2: $m_{r, t} \geqslant 3$, $w_m = \prescript{}{r}{\widehat{l}}u$ with $u$ minimal
   in its left $\langle r, t \rangle$-coset, possibly trivial and $2 \leqslant l \leqslant m_{r, t} - 1$.}
   Let $q \in \{r, t\}$ be such that $\prescript{}{r}{\widehat{l}} = \widehat{l}_q$ and 
   $q' \in \{r, t\} \setminus \{q\}$. By induction on the length of the final Garside factor, there exists some $\widetilde{w} \in \pi_s^{-1}(x)$
   such that $B_{\widetilde{w}}$ is an anchor of $F_{qu\beta}(B)$. Let $k$ be maximal such that 
   $\pH{k}(F_{qu\beta}(B))$ is non-zero. By definition of an anchor, there exists some $n \in \Z$ such that 
   \[\Hom_{K^b(\catFixedRightDesc{s})}(F_{qu\beta}(B), B_{\widetilde{w}}(n-k)[-n])\] does
   not vanish. Applying the autoequivalence $F_{\widehat{(l-1)}_{q'}}$ shows: 
   \[\Hom_{K^b(\catFixedRightDesc{s})}(F_{\sigma}(B), F_{\widehat{(l-1)}_{q'}} B_{\widetilde{w}}(n-k)[-n]) \neq 0 \]
   The claim in the proof of \cref{lemBraidStep} shows that there exist sets $I_t \subseteq \pi_s^{-1}(t)$
   and $I_r \subseteq \pi_s^{-1}(r)$ such that $F_{\widehat{(l-1)}_{q'}} B_{\widetilde{w}}$ is
   isomorphic to a perverse two-term minimal complex of the following form:
   \[ \dots \longrightarrow 0 \longrightarrow \bigoplus_{v \in I_r} B_v 
      \longrightarrow \bigoplus_{v \in I_t} B_v(1) \longrightarrow 0 \longrightarrow \dots \]
   where $\bigoplus_{v \in I_r} B_v$ sits in cohomological degree $0$. Since $u$ is minimal
   in its left $\langle r, t \rangle$-coset and $l < m_{r, t}$, it follows that $t \nin 
   \desc{L}(w_m)$. By A we know that there is no anchor of $F_{\sigma}(B)$ corresponding to $t$.
   Therefore a non-zero map from $F_{\sigma}(B)$ to $F_{\widehat{(l-1)}_{q'}} B_{\widetilde{w}}$ must
   have a non-zero component to a summand $B_w$ for $w \in I_r$ occuring in cohomological degree $0$. 
   Thus post-composition with the non-zero map \\
   \resizebox{\textwidth}{!}{$ \begin{xy}
     \xymatrix{
         {\dots} \ar[r] & 0 \ar[r] \ar[d] & {B_w(n-k) \oplus \bigoplus_{v \in I_r \setminus \{w\}} B_v(n-k)} \ar[r] 
	  \ar[d]^-{ \left( 
               \begin{smallmatrix}
                  \id, & 0, & \dots, & 0
               \end{smallmatrix} \right) }  
            & {\bigoplus_{v \in I_t} B_v(n+1-k)} \ar[r] \ar[d] 
            & 0 \ar[r] \ar[d] & {\dots} \\
         {\dots} \ar[r] & 0 \ar[r] & B_w(n-k) \ar[r] 
            & 0 \ar[r] & 0 \ar[r] & \dots }
     \end{xy} $}
   in $K^b(\catFixedRightDesc{s})$ shows that 
   $\Hom_{K^b(\catFixedRightDesc{s})}(F_{\sigma}(B), B_w(n-k)[-n])$ does not
   vanish. As before, we conclude that $B_w$ is an anchor of $F_{\sigma}(B)$.
   
   \emph{Case 3: $m_{r, t} \geqslant 3$, $w_m = \prescript{}{r}{(\widehat{m_{r,t}})}u$ with $u$ minimal
   in its left $\langle r, t \rangle$-coset and possibly trivial.}
   Let $q \in \{r, t\}$ be such that $\prescript{}{r}(\widehat{m_{r,t}}) = (\widehat{m_{r,t}})_q$ and 
   $q' \in \{r, t\} \setminus \{q\}$. We have $w_m = \prescript{}{r}{(\widehat{m_{r,t}})}_{q} u = 
   \prescript{}{t}{(\widehat{m_{r,t}})}_{q'} u$. By induction on the length of the final Garside 
   factor, we know that there exists $\widetilde{w} \in \pi_s^{-1}(q')$ such that $B_{\widetilde{w}}$ is
   an anchor of $F_{q'u\beta}(B)$. Let $k$ be maximal such that $\pH{k}(F_{q'u\beta}(B))$ is non-zero.  
   By definition of an anchor, there exists some $n \in \Z$ such that 
   \[\Hom_{K^b(\catFixedRightDesc{s})}(F_{q'u\beta}(B), B_{\widetilde{w}}(n-k)[-n])\] does not vanish.
   Since $F_{(\widehat{m_{r,t}-1})_q}B_{\widetilde{w}}$ is isomorphic to $B_w(1)[-1]$ for some 
   $w \in \pi_s^{-1}(r)$ by \cref{lemBraidStep} (due to $(\widehat{m_{r,t}-1})_q = 
   \prescript{}{t}{(\widehat{m_{r,t}-1})}$)  we can apply the autoequivalence 
   $F_{(\widehat{m_{r,t}-1})_q}$ to get a non-zero map from $F_{\sigma}(B)$ to $B_w(n+1-k)[-n-1])$. 
   Induction applied to A together with \cref{corNoAnchorMaxPH} show that $\pH{k}(F_{\sigma}(B))$ 
   is still the highest non-zero perverse cohomology group of $F_{\sigma}(B)$. 
   Therefore $B_w$ is an anchor of $F_{\sigma}(B)$.
   \end{proof}
   
   This concludes the proof of the main theorem.
\end{proof}

From this we will easily deduce the faithfulness of the action in finite type using \cref{lemBrmInjectivity}:

\begin{thm}
   \label{thmActFaithful}
   Assume that $(W, S)$ is a Coxeter system of finite type.
   The action of $\br[(W,S)]$ on $K^b(\catFixedRightDesc{s})$ is faithful.
\end{thm}
\begin{proof}
   Let $\rho: \br[(W, S)] \rightarrow Iso(Aut(K^b(\catFixedRightDesc{s})))$ be the group homomorphism
   corresponding to the action of $\br[(W, S)]$ on $K^b(\catFixedRightDesc{s})$ where 
   $Iso(Aut(K^b(\catFixedRightDesc{s})))$ is the group of isomorphism classes of autoequivalences 
   on $K^b(\catFixedRightDesc{s})$.
   Due to \cref{lemBrmInjectivity} it is enough to show the injectivity for its restriction $\rho^+$
   to the braid monoid $\brm{(W, S)}$. We will show that for any positive braid $\sigma \in \brm{(W, S)}$
   its Garside normal form and thus $\sigma$ itself can fully be recovered from the action 
   of $\sigma$ and its subwords on $K^b(\catFixedRightDesc{s})$. Set $B \defeq \bigoplus_{w 
   \in \lcellFixedRightDesc{s}} B_w$.
   
   Consider $F_{\sigma}(B)$ and its highest non-zero perverse cohomology group $\pH{m}(F_{\sigma}(B))$.
   \Cref{thmGarsideAct} implies that the number of Garside factors of $\sigma$ is $m$.
   Let $\sigma = w_m w_{m-1} \dots w_1$ be the Garside normal form of $\sigma$. To simplify notation,
   set $\beta = w_{m-1} \dots w_1$.
   
   To determine a reduced expression of $w_m$ we will proceed as follows: Let $B_{w}$ be
   an anchor in $\pH{m}(F_{\sigma}(B))$. \Cref{thmGarsideAct} shows that
   $\pi_s(w) = s_1$ lies in $\desc{L}(w_m)$. Write $w_m = s_1 u$ with $l(u) < l(w_m)$. Since the Rouquier complexes
   satisfy up to isomorphism the braid relations, we know: $E_{s_1} F_{\sigma} \cong F_{u\beta}$.
   Now consider the action of $F_{u \beta}$ on $B$. If $\pH{m}(F_{u\beta}(B))$ is zero, then
   $u = 1$, $u\beta$ has a Garside normal form with $m-1$ Garside factors and $w_m = s_1$.
   Otherwise we repeat the argument above to find an element in the left descent set of $u$.   
   After a finite number of steps we have reconstructed a reduced expression $w_m = s_1 s_2 \dots s_k$.
   
   By repeating the whole process we will eventually find all Garside factors of $\sigma$
   and we know their order. Thus we have determined $\sigma$ itself.
\end{proof}

The following result is an immediate consequence as the faithful action from \cref{thmActFaithful}
factors over the $2$-braid group.

\begin{corlab}[Faithfulness of the $2$-braid group in finite type]
   \label{corFaithfulness}
   Let $(W, S)$ be a Coxeter group of finite type.
   For two distinct braids $\sigma \neq \beta \in \br{(W, S)}$ the corresponding Rouquier complexes
   $F_{\sigma}$ and $F_{\beta}$ in the $2$-braid group are non-isomorphic.
\end{corlab}

Since in arbitrary type we can still differentiate the images of different positive braids 
in $\br{(W, S)}$, we get the following known result (the main result of \cite[Theorem 1.1]{Pa}):

\begin{corlab}[Injectivity of the canonical map]
   \label{corCanInject}
   Let $(W, S)$ be a Coxeter group of arbitrary type.
   The canonical morphism $\brm{(W, S)} \longrightarrow \br{(W, S)}$ is injective.
\end{corlab}

\vfill


\printbibliography

\Address

\end{document}